\documentclass[11pt]{article}
\usepackage{amssymb,amsmath,epsfig,amsthm}
\usepackage{fullpage,multirow}
\usepackage{tabu}
\usepackage{slashbox}

\usepackage{graphicx}
\usepackage{tikz}
\tikzset{partition/.style={fill,circle,inner sep=1pt}}
\usetikzlibrary{positioning}
\usetikzlibrary{decorations.pathreplacing}
\usetikzlibrary{matrix,arrows}
\usetikzlibrary{calc}
\tikzset{partition/.style={fill,circle,inner sep=1pt},
         part/.style={baseline=0,scale=0.5,bend left=45},
         partlabel/.style={below}}
\usetikzlibrary{shapes,arrows}
\tikzstyle{pnt}=[draw,ellipse,fill,inner sep=1pt]
\tikzstyle{opnt}=[ ]
\tikzstyle{pntt}=[draw,ellipse,fill,inner sep=0.5pt]
\tikzstyle{point}=[draw,ellipse,fill,inner sep=2pt]
\usepackage{color}

\newtheorem{theorem}{Theorem}[section]
\newtheorem{prop}[theorem]{Proposition}
\newtheorem{lemma}[theorem]{Lemma}

\newtheorem{question}[theorem]{Question}
\theoremstyle{definition}

\newtheorem{definition}{Definition}

\newcommand\beq{\begin{equation}}
\newcommand\eeq{\end{equation}}
\newcommand\bce{\begin{center}}
\newcommand\ece{\end{center}}
\newcommand\bea{\begin{eqnarray}}
\newcommand\eea{\end{eqnarray}}
\newcommand\bean{\begin{eqnarray*}}
\newcommand\eean{\end{eqnarray*}}
\newcommand\bmt{\begin{multline*}}
\newcommand\emt{\end{multline*}}
\newcommand\ben{\begin{enumerate}}
\newcommand\een{\end{enumerate}}
\newcommand\bit{\begin{itemize}}
\newcommand\eit{\end{itemize}}
\newcommand\brr{\begin{array}}
\newcommand\err{\end{array}}
\newcommand\bt{\begin{tabular}}
\newcommand\et{\end{tabular}}
\newcommand\ba{\begin{array}}
\newcommand\ea{\end{array}}

\newcommand\ms{\medskip}

\renewcommand\S{\mathcal S}
\renewcommand\L{\mathcal L}

\newcommand\Lu{{\mathcal L}^\times}

\newcommand\K{\mathcal K}
\newcommand\M{\mathcal M}
\renewcommand\P{\mathcal P}
\newcommand\F{\mathcal F}
\newcommand\E{\mathcal E}
\newcommand\A{\mathcal A}
\newcommand\B{\mathcal B}
\def \R{\mathcal{R}}
\def \D{\mathcal{D}}
\def \set #1{\{#1\}}

\DeclareMathOperator{\val}{val}
\DeclareMathOperator{\peak}{peak}
\newcommand{\hgt}{h}

\date{}

\begin{document}
\title{Pattern avoidance in matchings and partitions}
\author{Jonathan Bloom and Sergi Elizalde~\thanks{Department of Mathematics,
Dartmouth College, Hanover, NH 03755.}}
\maketitle

\begin{abstract}
Extending the notion of pattern avoidance in permutations, we study matchings and set partitions whose arc diagram representation avoids a given configuration of three arcs.
These configurations, which generalize $3$-crossings and $3$-nestings, have an interpretation, in the case of matchings, in terms of patterns in full rook placements on Ferrers boards.

We enumerate $312$-avoiding matchings and partitions, obtaining algebraic generating functions, in contrast with the known D-finite generating functions for the $321$-avoiding (i.e., $3$-noncrossing) case.
Our approach also provides a more direct proof of a formula of B\'ona for the number of $1342$-avoiding permutations.
Additionally, we give a bijection proving the shape-Wilf-equivalence of the patterns $321$ and $213$ which greatly simplifies existing proofs by Backelin--West--Xin and Jel\'{\i}nek,
and provides an extension of work of Gouyou-Beauchamps for matchings with fixed points.
Finally, we classify pairs of patterns of length 3 according to shape-Wilf-equivalence, and enumerate matchings and partitions avoiding a pair in most of the resulting equivalence classes.

\end{abstract}

\section{Introduction}

Pattern avoidance in matchings is a natural extension of pattern avoidance in permutations. Indeed, a permutation of $[n]=\{1,2,\dots,n\}$ can be thought of as matching
of $[2n]$ where each element of $[n]$ is paired up with an element of $[2n]\setminus[n]$.
The natural translation of the definition of patterns in permutations to this type of matchings
extends to all perfect matchings, and more generally, to set partitions ---which, when all the blocks have size 2, are just perfect matchings.
We will use the term matching to refer to a perfect matching, when it creates no confusion.
On the other hand, the well-studied notions of $k$-crossings and $k$-nestings in matchings and set partitions,
in our language, are simply occurrences of the patterns $k\dots 21$ and $12\dots k$, respectively.
Additionally, by viewing matchings as certain fillings of Ferrers boards, patterns in matchings relate to patterns in Ferrers boards, and thus to the concept of shape-Wilf-equivalence of permutations.

Motivated by these connections and by the recent work on crossings, nestings, permutation patterns, and shape-Wilf-equivalence, we study matchings and partitions that avoid patterns of length 3.
We consolidate and simplify recent work on the classification of these patterns, and we obtain new results on the enumeration of matchings and partitions that avoid some of these patterns.

\subsection{Previous work}

Given $2n$ points on a horizontal line, labeled increasingly from left to right, we represent a matching of $[2n]$
by drawing $n$ arcs between pairs of points. If $i<j<k<\ell$, two arcs $(i,k)$, $(j,\ell)$ form a crossing, and two arcs $(i,\ell)$, $(j,k)$ form a nesting.
Similarly, a partition of $[n]$ is represented by drawing, for each block $\{i_1,i_2,\dots,i_a\}$ of size $a$ with $i_1<i_2<\dots<i_a$, $a-1$ arcs $(i_1,i_2),(i_2,i_3),\dots,(i_{a-1},i_a)$.
A crossing in the partition is then a pair of arcs $(i,k)$, $(j,\ell)$, and a nesting is a pair of arcs $(i,\ell)$, $(j,k)$, where $i<j<k<\ell$.

Crossings and nestings in matchings and partitions have been studied for decades.
It is well known that the number of perfect matchings on $[2n]$ with no crossings (or with no nestings)
is the $n$-th Catalan number $C_n$,
which also equals the number of partitions of $[n]$ of with no crossings, and the number of those with no nestings.

More generally, attention has focused on the study of $k$-crossings ($k$-nestings), which are sets of $k$ pairwise crossing (respectively, nesting) arcs.
For set partitions, the above definition, which we use throughout the paper, is the same given by Chen, Deng, Du, Stanley and Yan~\cite{CDDSY} and Krattenthaler~\cite{Kra06}. However,
we point out that different definitions of pattern avoidance for partitions have been introduced by Klazar~\cite{Kla} and Sagan~\cite{Sag}, studied also in~\cite{ManSev,JelMan,Goyt}.

Touchard~\cite{Tou} and Riordan~\cite{Rio} considered the distribution of $2$-crossings on matchings, which was shown~\cite{SC} to be equal to the distribution of $2$-nestings.
The number of $3$-nonnesting matchings of $[2n]$ (viewed as fixed-point-free involutions with no decreasing sequence of length 6) was found by Gouyou-Beauchamps~\cite{Gou},
who recursively constructed a bijection onto pairs of noncrossing Dyck paths with $2n$ steps, counted by $C_nC_{n+2}-C_{n+1}^2$.
More recently, Chen et al.~\cite{CDDSY} showed that, if we define the crossing (nesting) number of a matching as the maximum $k$ such that it contains a $k$-crossing (resp. $k$-nesting), then the
crossing number and the nesting number have a symmetric joint distribution on the set of all matchings of $[2n]$. In particular, the number of $k$-noncrossing matchings
(i.e., containing no $k$-crossing) of $[2n]$ equals the number of $k$-nonnesting (i.e., containing no $k$-nesting) ones, for all $k$.
They show that the analogous results hold for partitions as well. Their proof, which uses vacillating tableaux and a variation of Robinson-Schensted insertion and deletion, also provides a bijection
between $k$-noncrossing matchings and certain $(k-1)$-dimensional closed lattice walks, from where a determinant formula for the generating function in terms of hyperbolic Bessel functions follows.

Less is known about the enumeration of $k$-noncrossing set partitions. Bousquet-M\'elou and
Xin~\cite{BMXin} settled the case $k=3$ using a bijection into lattice paths to derive a functional equation for the generating function, which then is solved by the kernel method. They showed that the generating
function for $3$-noncrossing set partitions is D-finite, that is, it satisfies a linear differential equation with polynomial coefficients.
This is conjectured not to be the case for $k>3$.
For $k$-nonnesting set partitions, additional functional equations for the generating functions have been obtained by Burrill et al.~\cite{BEMY} using generating trees for open arc diagrams.

By interpreting matchings and partitions as rook placements on Ferrers boards and using the growth diagram construction of Fomin~\cite{Fom86},
Krattenthaler~\cite{Kra06} gave a simpler description of the bijections in~\cite{CDDSY} proving the symmetry of crossing and nesting number on matchings
and partitions. He extended the results to fillings of Ferrers boards with nonnegative integers. Other extensions have been given by de Mier~\cite{Mier} to fillings with prescribed row and column sums.

As mentioned before, $k$-crossings (respectively, $k$-nestings) in matchings have a simple interpretation as occurrences of the monotone decreasing (respectively, increasing) pattern of length $k$.
In this paper we study and enumerate matchings that avoid other patterns of length~3, and in some cases, we extend our results to the enumeration of pattern-avoiding partitions.
The translation of crossings and nestings to the language of permutation patterns becomes natural via a bijection between matchings and certain fillings of Ferrers boards, called full rook placements, described in Section~\ref{sec:matchings}.
For such fillings, the definitions of pattern containment and avoidance in permutations generalize routinely, and they have been widely studied in the literature.
In this setting, Stankova and West~\cite{SW} introduced the concept of shape-Wilf-equivalence, and they showed that the patterns $231$ and $312$ are shape-Wilf-equivalent. A simpler proof of this fact was later given by Bloom and Saracino~\cite{BloSar}.
As we will see, if two patterns are shape-Wilf-equivalent, then the number of matchings avoiding one is the same as the number of those avoiding the other, and the same is true for partitions.
Backelin, West and Xin~\cite{BWX} showed that $12\dots k$ and $k\dots21$ are shape-Wilf-equivalent. A more direct proof of their result, which implies again that
$k$-nonnesting and $k$-noncrossing matchings are equinumerous, was given by Krattenthaler~\cite{Kra06}.
It also follows from~\cite{BWX} that $123$ and $213$ are shape-Wilf-equivalent. Thus, there are three shape-Wilf-equivalence classes of patterns of length 3, namely $123\sim321\sim213$, $231\sim312$, and $132$.

Jel\'{\i}nek~\cite{Jel} reproved some of these results independently in the context of matchings, by giving
bijections between $231$-avoiding matchings and $312$-avoiding ones, and
between $213$-avoiding matchings and $123$-avoiding (i.e. $3$-nonnesting) ones.

Finally, let us mention that Stankova~\cite{Stankova} compared, for each one of the three shape-Wilf-equivalence classes of patterns of length 3,
the number of full rook placements on any given Ferrers board avoiding each a pattern in the class. She showed that the number of $231$-avoiding placements is no larger than the number of $321$-avoiding placements
(this is also proved in~\cite{Jel}), which is in turn no larger than the number of $132$-avoiding ones.

\subsection{Structure of the paper}

In Section~\ref{sec:defs} we define patterns in matchings, in set partitions, and in rook placements on Ferrers boards,
and we set the notation for the rest of the paper.
In Sections \ref{sec:123}, \ref{sec:231} and~\ref{sec:132} we study each one of
the three shape-Wilf-equivalence classes of patterns of length~3.
In Section~\ref{sec:123} we give a new simple bijection between $123$-avoiding matchings and $213$-avoiding ones, 
as well as a generalization to matchings with fixed points (i.e., not necessarily perfect) which relates to the  main result from~\cite{Gou}.
In Section~\ref{sec:231} we enumerate $231$-avoiding (equivalently, $312$-avoiding) matchings and partitions, and we show that their generating functions are algebraic, in contrast to the case of $123$-avoiding matchings~\cite{Gou} and partitions~\cite{BMXin}. We then use our techniques for matchings to obtain a new proof of B\'ona's formula for the generating function for $1342$-avoiding permutations~\cite{Bona}.
In Section~\ref{sec:132} we discuss $132$-avoiding matchings, for which no enumeration formula is known, and we argue that counting them
is closely related to counting $1324$-avoiding permutations, which is an outstanding open problem.
Finally, in Section~\ref{sec:double} we enumerate matchings and partitions that avoid pairs of patterns of length~3.

\section{Matchings, partitions, and rook placements}\label{sec:defs}

\subsection{Ferrers boards}\label{sec:Ferrers}

A \emph{Ferrers board} is a left-justified array of unit squares so that the number of squares in each row is less than or equal to the number of squares in the row below.
To be precise, consider an $n\times n$ array of unit squares in the $xy$-plane, whose bottom left corner is at the origin $(0,0)$. The vertices of the unit squares are lattice points in $\mathbb{Z}^2$.
For any vertex $V=(a,b)$, let $\Gamma(V)$ be the set of unit squares inside the rectangle $[0,a]\times[0,b]$.
Then, a subset $F$ of the $n\times n$ array with the property that $\Gamma(V)\subseteq F$ for each vertex in $F$ is a Ferrers board. Equivalently, $F$ is bounded by the coordinate lines and by a lattice path from $(0,n)$ to $(n,0)$ with east steps $(1,0)$ and south steps $(0,-1)$. We call this path the \emph{border} of $F$, and we denote its vertices by
$V_0, \ldots, V_{2n}$, where $V_0=(0,n)$, $V_n=(n,0)$ and $V_{i+1}$ is immediately below or to the right of $V_i$.

\begin{definition}
A \emph{full rook placement} is a pair $(R,F)$ where $F$ is a Ferrers board and $R$ is a subset of squares of $F$ (marked by placing a rook in each one of them) such that each row and each column of $F$ contains exactly one rook.
Let $\R_F$ denote the set of full rook placements on $F$.
\end{definition}

Figure~\ref{fig:RookPlacement} gives an example of a full rook placement, where rooks are indicated by the symbol~$\times$.
In this paper, the term {\em placement} will always refer to a full rook placement.
For a Ferrers board $F$ to admit a full rook placement, the number or non-empty rows must equal the number
of non-empty columns, and the coordinates $(x,y)$ of the vertices in the border of $F$ must satisfy $x\ge y$.
We denote by $\F_n$ the set of Ferrers boards satisfying this condition and having $n$ non-empty rows and columns.
The border of $F\in\F_n$, which we denote by $D_F$, is a lattice path from $(0,n)$ to $(n,0)$ with steps east and south that remains above the line $y = n-x$.
We denote by $\D_n$ the set of such paths, which we call {\em Dyck paths} of semilength $n$ (despite being rotated from other standard ways of drawing them).
The map $F\mapsto D_F$ is a trivial bijection between $\F_n$ and $\D_n$.  A \emph{peak} on a Dyck path $D$ is an occurrence of an east step immediately followed by a south step.  Likewise, a \emph{valley} is an occurrence of a south step immediately followed by an east step.  We write $\peak(D)$ and $\val(D)$ for the number of peaks and valleys on $D$ respectively, and note that $\peak(D)=\val(D)+1$ for $D\in\D_n$ with $n\ge1$.
Analogously, we define a \emph{peak} (respectively, a \emph{valley}) on $F\in \F_n$ to be a vertex $V_i$  that is above (respectively, to the right) $V_{i+1}$ and to the left of (respectively, below) $V_{i-1}$. We write $\peak(F)$ and $\val(F)$ to denote the number of peaks and valley on $F$ respectively. Clearly, $\peak(F)=\peak(D_F)$ and $\val(F)=\val(D_F)$.

We let
$$\R_n = \bigcup_{F\in\F_n} \R_F.$$ 
be the set of all placements on boards in $\F_n$.
Denote by $\S_n$ the set of permutations of $\{1,2,\dots,n\}$. To each full rook placement $(R,F)$ where $F\in\F_n$, one can associate a permutation $\pi_R\in\S_n$
by letting $\pi_R(i)=j$ if $R$ has a rook in column $i$ and row $j$ (our convention here is to number the columns of $F$ from left to right and its rows from bottom to top, as in the usual cartesian coordinates).
In the case that $F\in\F_n$ is the  square Ferrers board, this map is a bijection between full rook placements on $F$ and $\S_n$.
More generally, given a vertex $V$ of the border of $F$, the restriction of the placement $R$ to the rectangle $\Gamma(V)$, which consits of the squares $R\,\cap\,\Gamma(V)$,
determines a unique permutation in $\S_k$, where $k=|R\,\cap\,\Gamma(V)|$. This permutation is obtained
by disregarding empty rows and columns, and then applying the above map. Under this correspondence it makes sense to consider concepts such as the longest increasing sequence in $R\,\cap\,\Gamma(V)$.

\begin{figure}[h!]
\begin{center}
\begin{tikzpicture}[scale=0.4]

\begin{scope}[shift={(0,0)}]
\draw[fill=gray!20] (0,0) rectangle (5,6);

\draw (0,0) -- (0,8);
\draw (1,0) -- (1,8);
\draw (2,0) -- (2,8);
\draw (3,0) -- (3,8);
\draw (4,0) -- (4,7);
\draw (5,0) -- (5,6);
\draw (6,0) -- (6,5);
\draw (7,0) -- (7,5);
\draw (8,0) -- (8,4);

\draw (0,0) -- (8,0);
\draw (0,1) -- (8,1);
\draw (0,2) -- (8,2);
\draw (0,3) -- (8,3);
\draw (0,4) -- (8,4);
\draw (0,5) -- (7,5);
\draw (0,6) -- (5,6);
\draw (0,7) -- (4,7);
\draw (0,8) -- (3,8);

\node at (0.5,0.5) {$\times$};
\node at (1.5,6.5) {$\times$};
\node at (2.5,7.5) {$\times$};
\node at (3.5,5.5) {$\times$};
\node at (4.5,2.5) {$\times$};
\node at (5.5,1.5) {$\times$};
\node at (6.5,4.5) {$\times$};
\node at (7.5,3.5) {$\times$};
\node at (6.5,7.5) {$V$};
\draw[->] (6.1,7.1) -- (5.1,6.1);
\draw[line width = .5mm] (0,8) -- (3,8) --(3,7) -- (4,7) -- (4,6) -- (5,6) -- (5,5) -- (7,5) -- (7,4) -- (8,4) -- (8,0);
\end{scope}
\end{tikzpicture}
\caption{A rook placement $(R,F)$ with $F\in\F_8$ and $\pi_R = 17863254$. For the selected vertex $V$ on the border (the thicker path), $\Gamma(V)$ is the shaded rectangle and $R\,\cap\, \Gamma(V)$ determines the permutation $132$.}
\label{fig:RookPlacement}
\end{center}
\end{figure}
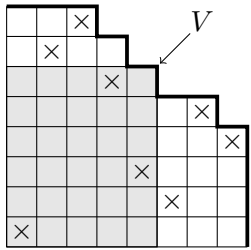

Recall that a permutation $\pi\in\S_n$ \emph{avoids} another permutation $\tau\in\S_k$ (usually called a \emph{pattern}) if there is no subsequence $\pi(i_1)\dots\pi(i_k)$ with $i_1<\dots<i_k$ that is order-isomorphic to $\tau(1)\dots\tau(k)$.
The number of $\tau$-avoiding permutations in $\S_n$ is denoted by $\S_n(\tau)$.
Viewing permutations as full rook placements on the square Ferrers board, $\pi$ avoids $\tau$ if the placement corresponding to $\tau$ cannot be obtained from the
placement corresponding to $\pi$ by removing rows and columns. This definition has been generalized \cite{BWX} to rook placements as follows.

\begin{definition}\label{def:avoidanceF}
A full rook placement $(R,F)$ \emph{avoids} $\tau\in\S_k$ if and only if for every vertex $V$ on the border of $F$, the permutation given by $R\,\cap\,\Gamma(V)$ avoids $\tau$. Let $\R_F(\tau)$ be the set of full rook placements on $F$ that avoid $\tau$. Similarly, let
$$\R_n(\tau)= \bigcup_{F\in\F_n} \R_F(\tau).$$
\end{definition}

\begin{definition}
Two patterns $\sigma$ and $\tau$ are said to be \emph{shape-Wilf-equivalent}, denoted $\sigma\sim\tau$, if for any Ferrers board $F$ we have $|\R_F(\sigma)|=|\R_F(\tau)|$.
\end{definition}

Clearly, if two patterns are shape-Wilf-equivalent, then they are also Wilf-equivalent, meaning that they are avoided by the same number of {\em permutations}.
The converse is not true, as shown by the fact that there is one Wilf-equivalence class for patterns of length 3, but three shape-Wilf-equivalence classes:
$123\sim321\sim213$ (see~\cite{BWX,Kra06}), $231\sim312$ (see~\cite{SW,Jel,BloSar}), and $132$.

We point out that there are two definitions of shape-Wilf-equivalence in the literature, one of which is obtained by complementation of the other.
This arises from the fact that Ferrers boards can be drawn in French notation or English notation, depending on whether column widths weakly decrease from bottom to top or from top to bottom, respectively,
and also the fact that entry $(i,\pi(i))$ can describe cartesian coordinates ($i$-th column from the left, $\pi(i)$-th row from the bottom)
or matrix coordinates ($i$-th row from the top, $\pi(i)$-th column from the left). Our convention of using French notation and cartesian coordinates gives the same definition
of shape-Wilf-equivalence from~\cite{BWX}, which uses English notation and matrix coordinates, and~\cite{BMSte,BloSar}, which use the same conventions as in this paper.
However, the definition from~\cite{SW,Stankova,Jel}, which uses English notation and cartesian coordinates, would give the complemented shape-Wilf-equivalence classes $321\sim123\sim231$, $213\sim132$, and $312$.

Regarding shape-Wilf-Equivalence of patterns of arbitrary length, two important results are due to Backelin, West and Xin~\cite{BWX}. One states that $12\dots k\sim k\dots 21$ for all $k$, and the other one is the following.

\begin{prop}[\cite{BWX}]
Let $\sigma,\tau\in\S_k$ and $\rho\in \S_\ell$. If $\sigma\sim\tau$, then $\sigma\rho'\sim \tau\rho'$, where $\rho'$ is obtained from $\rho$ by adding $k$ to each of its entries.
\end{prop}

Denote by $\D_n^2$ the set of pairs $(D_0,D_1)$ of Dyck paths $D_0,D_1\in\D_n$ such that $D_0$ never goes above $D_1$. We say that $D_0$ and $D_1$ are \emph{noncrossing}, and we call $D_0$ the \emph{bottom path}
and $D_1$ the \emph{top path}.  For any $F\in \F_n$, we denote by $\D_F^2$ the set of pairs $(D_0,D_F)\in \D_n^2$, that is, those where the top path is the border of $F$.

\subsection{Matchings}\label{sec:matchings}

Denote by $\M_n$ the set of perfect matchings on $[2n]$. Recall that a perfect matching is a set of $n$ pairs $(i,j)$, with $i<j$ such that each element in $[2n]$ appears in exactly one of the pairs. If $(i,j)$ is such a pair,
we say that vertices $i$ and $j$ are matched, and we call $i$ an \emph{opener} and $j$ a \emph{closer}.
As mentioned before, we will use the term matching to mean perfect matching.
We represent matchings as arc diagrams as follows: place $2n$ equally spaced points on a horizontal line, numbered from left to right,
and draw an arc between the two vertices of each pair. The picture on the left of Figure~\ref{fig:kappa} corresponds to the matching $(1,6),(2,12),(3,4),(5,7),(8,10),(9,11)$.

The following natural bijection between $\M_n$ and $\R_n$, which we denote $\kappa$, has been used
in~\cite{Mier,Jel}. Given a matching $M\in\M_n$, construct a path from $(0,n)$ to $(n,0)$ by reading the
vertices of $M$ in increasing order, and adding an east step $(1,0)$ for each opener, and a south step $(0,-1)$ for each closer. This path is clearly a Dyck path, so it is the border of a Ferrers board $F\in\F_n$, which we call the
\emph{shape} of $M$. Each column of $F$ is naturally associated to an opener of $M$ (the vertex that produced the east step at the top of the column), and each row is naturally associated to a closer. Now define a full rook $R$ placement on $F$ by placing a rook in the column associated to $i$ and the row associated to $j$ for each matched pair $(i,j)$. An example of the bijection $\kappa:M\mapsto(R,F)$ is given in Figure~\ref{fig:kappa}.

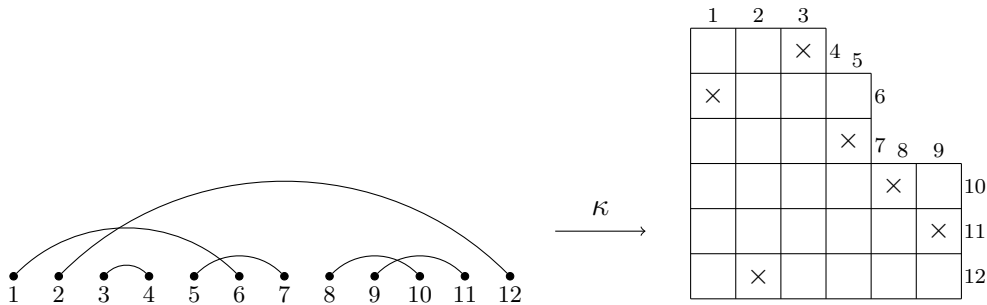
\begin{figure}[h!]
\begin{center}
\begin{tikzpicture}[scale=0.6]
\foreach \i in {0,...,11}
	{
		\node[pnt] at (\i,0)(\i){};
	}
	\foreach \i / \j in {0/1,1/2,2/3, 3/4, 4/5, 5/6, 6/7, 7/8,8/9,9/10,10/11,11/12}
	{
		\draw (\i) node[below] {\footnotesize \j};
	}
\draw(0)  to [bend left=45] (5);
\draw(1)  to [bend left=45] (11);
\draw(2)  to [bend left=45] (3);
\draw(4)  to [bend left=45] (6);
\draw(7)  to [bend left=45] (9);
\draw(8)  to [bend left=45] (10);
\draw[->] (12,1) -- (14,1);
\draw (13,1.5) node {\large $\kappa$};

\begin{scope}[shift={(15,-.5)}]

\foreach \i / \h in {0/6,1/6, 2/6, 3/6, 4/5, 5/3, 6/3}
	{
		\draw (\i, 0) -- (\i, \h);
	}
\foreach \i / \w in {0/6,1/6, 2/6, 3/6, 4/4, 5/4, 6/3}
	{
		\draw (0, \i) -- (\w, \i);
	}
\foreach \i / \w in {0.5/4.5,1.5/0.5, 2.5/5.5, 3.5/3.5, 4.5/2.5, 5.5/1.5}
	{
		\draw (\i, \w) node {{$\times$}};
	}
	
\draw(0.5,6.3) node {\scriptsize 1};
\draw(1.5,6.3) node {\scriptsize 2};
\draw(2.5,6.3) node {\scriptsize 3};
\draw(3.2,5.5) node {\scriptsize 4};
\draw(3.7,5.3) node {\scriptsize 5};
\draw(4.2,4.5) node {\scriptsize 6};
\draw(4.2,3.4) node {\scriptsize 7};
\draw(4.7,3.3) node {\scriptsize 8};
\draw(5.5,3.3) node {\scriptsize 9};
\draw(6.3,2.5) node {\scriptsize 10};
\draw(6.3,1.5) node {\scriptsize 11};
\draw(6.3,0.5) node {\scriptsize 12};

\end{scope}
\end{tikzpicture}
\caption{An example of the bijection $\kappa:\M_{n}\rightarrow\R_n$.}
\label{fig:kappa}
\end{center}
\end{figure}

In light of this bijection, the definition of pattern avoidance in Ferrers boards translates naturally to matchings.

\begin{definition}
We say that a matching $M\in\M_{n}$ {\em avoids} the pattern $\tau\in\S_k$ if the corresponding full rook placement $\kappa(M)$ does.
Equivalently, directly in terms of matchings, $M$ avoids $\tau$ if
there are no $2k$ vertices $1\le i_1 < \ldots < i_{2k} \le n$ such that $M$ contains all the pairs $(i_{a},i_{2k+1-\tau(a)})$ for $1\le a\le k$.

For fixed $F\in\F_n$, denote by $\M_F  = \kappa^{-1}(\R_F)$ the set of matchings of shape $F$, and by
$\M_F(\tau) = \kappa^{-1}(\R_F(\tau))$ those that avoid $\tau$.
Analogously, denote by $\M_n(\tau) = \kappa^{-1}(\R_n(\tau))$ the set of $\tau$-avoiding matchings in $\M_{n}$. Note that
$$\M_n = \bigcup_{F\in\F_n}\; \M_F \qquad\text{and}\qquad \M_n(\tau) = \bigcup_{F\in\F_n}\; \M_F(\tau).$$
\end{definition}

This definition extends the notions of $k$-noncrossing and $k$-nonnesting matchings studied in~\cite{CDDSY,Kra06}.
Recall that a matching is $k$-nonncrossing if it contains no $k$ mutually crossing arcs. In our terminology, this is equivalent to avoiding
the pattern $k\dots21$. Similarly, a matchings is $k$-nonnesting if it contains no $k$ mutually crossing arcs, which is equivalent to
avoiding $12\dots k$.

For patterns $\tau\in\S_3$, which are the focus of this paper, we can describe $\M_n(\tau)$
as the set of matchings $M\in\M_n$ containing no three arcs whose endpoints occur in the same order as in the corresponding configuration in Figure~\ref{fig:arcpatterns}.

\begin{figure}[h!]
\begin{center}
\begin{tikzpicture}[scale=0.5]
\begin{scope}[shift={(0,4)}]
	\foreach \i in {0,...,5}
	{
		\node[pnt] at (\i,0)(\i){};
	}
	\draw(0)  to [bend left=45] (3);
	\draw(1)  to [bend left=45] (4);
	\draw(2)  to [bend left=45] (5);
	\draw (2.5,-1) node {$321$};
	
	\begin{scope}[shift={(8,0)}]
		\foreach \i in {0,...,5}
		{
			\node[pnt] at (\i,0)(\i){};
		}
		\draw(0)  to [bend left=45] (5);
		\draw(1)  to [bend left=45] (4);
		\draw(2)  to [bend left=45] (3);
		\draw (2.5,-1) node {$123$};
	\end{scope}
	
	\begin{scope}[shift={(16,0)}]
		\foreach \i in {0,...,5}
		{
			\node[pnt] at (\i,0)(\i){};
		}		
		\draw(1)  to [bend left=45] (3);
		\draw(2)  to [bend left=45] (4);
		\draw(0)  to [bend left=45] (5);
		\draw (2.5,-1) node {$132$};
	\end{scope}
\end{scope}

\begin{scope}[shift={(0,0)}]
	\foreach \i in {0,...,5}
		{
			\node[pnt] at (\i,0)(\i){};
		}		
		\draw(1)  to [bend left=45] (3);
		\draw(0)  to [bend left=45] (4);
		\draw(2)  to [bend left=45] (5);
		\draw (2.5,-1) node {$231$};
\end{scope}
\begin{scope}[shift={(8,0)}]
	\foreach \i in {0,...,5}
		{
			\node[pnt] at (\i,0)(\i){};
		}		
		\draw(0)  to [bend left=45] (3);
		\draw(2)  to [bend left=45] (4);
		\draw(1)  to [bend left=45] (5);
		\draw (2.5,-1) node {$312$};
\end{scope}
\begin{scope}[shift={(16,0)}]
	\foreach \i in {0,...,5}
		{
			\node[pnt] at (\i,0)(\i){};
		}		
		\draw(2)  to [bend left=45] (3);
		\draw(0)  to [bend left=45] (4);
		\draw(1)  to [bend left=45] (5);
		\draw (2.5,-1) node {$213$};
\end{scope}
\end{tikzpicture}
\caption{Forbidden configurations corresponding to $\tau\in \S_3$.}
\label{fig:arcpatterns}
\end{center}
\end{figure}
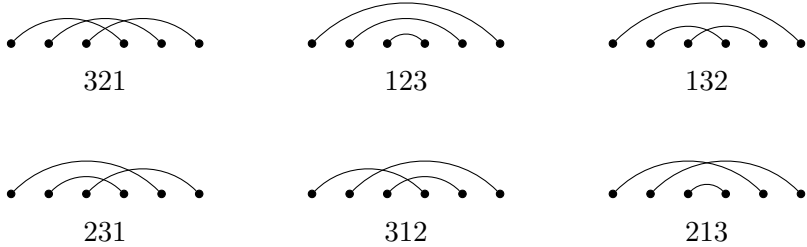

Since $\kappa$ is a bijection, it is clear that $|\M_F(\tau)|=|\R_F(\tau)|$ for any $\tau$. Thus, shape-Wilf-equivalence can be interpreted in terms of pattern-avoiding matchings:
$\sigma\sim\tau$ if and only if $|\M_F(\sigma)|=|\M_F(\tau)|$ for every Ferrers board $F$ (or equivalently, for every $F\in\bigcup_n\F_n$, since the condition
is void otherwise). In particular, if $\sigma\sim\tau$, then $|\M_n(\sigma)|=|\M_n(\tau)|$ for all $n$. No counterexample for the converse statement is known.

\begin{question}\label{quest:converse}
Are there patterns $\sigma$, $\tau$ that satisfy $|\M_n(\sigma)|=|\M_n(\tau)|$ for all $n$, but are not shape-Wilf-equivalent?
\end{question}

In the case of simultaneous avoidance of a pair of patterns, we answer the above question in the affirmative in Section~\ref{sec:IIandIII}.

\subsection{Set partitions}\label{sec:partitions}

Denote by $\P_n$ the set of partitions of $[n]$. Similarly to what we did for matchings, we represent
partitions of $[n]$ as arc diagrams on
$n$ points on a horizontal line, numbered from left to right. For each block $\{i_1,i_2,\dots,i_j\}$ with $i_1<i_2<\dots<i_j$, we draw $j-1$ arcs
$(i_1,i_2), (i_2,i_3), \dots, (i_{j-1},i_j)$ (see Figure~\ref{fig:setpartition}).
If $j\ge2$, we call $i_1$ an \emph{opener}, $i_j$ a \emph{closer}, and we say that $i_2,\dots,i_{j-1}$ are \emph{transitory} vertices. If $j=1$,
the vertex $i_1$ is called a singleton.

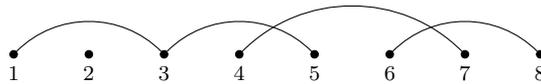
\begin{figure}[h!]
\begin{center}
\begin{tikzpicture}[scale=0.5]

	\node[pnt] at (0,0)(1){};
	\node[pnt] at (2,0)(2){};
	\node[pnt] at (4,0)(3){};
	\node[pnt] at (6,0)(4){};
	\node[pnt] at (8,0)(5){};
	\node[pnt] at (10,0)(6){};
	\node[pnt] at (12,0)(7){};
	\node[pnt] at (14,0)(8){};
	
	\node at (0,-.5){\scriptsize 1};
	\node at (2,-.5){\scriptsize 2};
	\node at (4,-.5){\scriptsize 3};
	\node at (6,-.5){\scriptsize 4};
	\node at (8,-.5){\scriptsize 5};
	\node at (10,-.5){\scriptsize 6};
	\node at (12,-.5){\scriptsize 7};
	\node at (14,-.5){\scriptsize 8};
	
	\draw(1)  to [bend left=45] (3);
	\draw(3)  to [bend left=45] (5);
	\draw(4)  to [bend left=45] (7);
	\draw(6)  to [bend left=45] (8);

	\end{tikzpicture}
\caption{The arc diagram of the partition $\{\{1,3,5\},\{2\},\{4,7\},\{6,8\}\}$. Vertices $1,4,6$ are openers, $5,7,8$ are closers, $3$ is transitory, and $2$ is a singleton.}
\label{fig:setpartition}
\end{center}
\end{figure}

We will use the term partition to refer to a set partition when it creates no confusion.
Note that matchings are partitions where all blocks have size $2$. The definition of pattern avoidance for matchings extends to partitions as follows.

\begin{definition}
We say that a partition $P\in\P_{n}$ {\em avoids} the pattern $\tau\in\S_k$ if
there are no $2k$ vertices $1\le i_1 < \ldots < i_{2k} \le n$ such that $P$ contains all the arcs $(i_{a},i_{2k+1-\tau(a)})$ for $1\le a\le k$.
Denote by $\P_n(\tau)$ the set of $\tau$-avoiding partitions in $\P_{n}$.
\end{definition}

Note that in the above definition, singleton blocks of $P$ do not contribute to occurrences of any pattern~$\tau$.

To enumerate partitions avoiding a pattern, we will use the following construction that associates a matching to each partition.
Given a partition $P$ represented as an arc diagram, remove all singleton vertices, and split each transitory vertex into two vertices: a closer followed by an opener (see Figure~\ref{fig:partition_to_matching}).

\begin{figure}[h!]
\begin{center}
\begin{tikzpicture}[scale=0.5]
	\node[pnt] at (0,0)(1){};

	\node[pnt] at (3.7,0)(3){};
	\node[pnt] at (4.3,0)(3'){};
	\node[pnt] at (6,0)(4){};
	\node[pnt] at (8,0)(5){};
	\node[pnt] at (10,0)(6){};
	\node[pnt] at (12,0)(7){};
	\node[pnt] at (14,0)(8){};
	
	\node at (0,-.5){\scriptsize 1};

	\node at (3.7,-.5){\scriptsize $3'$};
	\node at (4.3,-.5){\scriptsize $3''$};
	\node at (6,-.5){\scriptsize 4};
	\node at (8,-.5){\scriptsize 5};
	\node at (10,-.5){\scriptsize 6};
	\node at (12,-.5){\scriptsize 7};
	\node at (14,-.5){\scriptsize 8};
	
	\draw(1)  to [bend left=45] (3);
	\draw(3')  to [bend left=45] (5);
	\draw(4)  to [bend left=45] (7);
	\draw(6)  to [bend left=45] (8);\end{tikzpicture}
\caption{The matching associated to the partition in Figure~\ref{fig:setpartition}. Vertex $2$ has been removed, and vertex $3$ has been split into two vertices $3'$ and $3''$.}
\label{fig:partition_to_matching}
\end{center}
\end{figure}
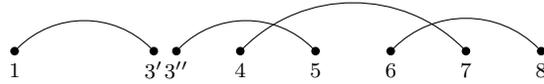

This operation produces a matching $M_P$, and the transformation $P\mapsto M_P$ preserves occurrences of every pattern $\tau$. In particular,
$P$ avoids $\tau$ if and only if $M_P$ does. If $P\in\P_n$ has $b$ blocks, then $M_P\in\M_{n-b}$.

The above construction can be reversed to generate all $\tau$-avoiding set partitions $\bigcup_n\P_n(\tau)$ from the set of all $\tau$-avoiding matchings $\bigcup_n\M_n(\tau)$.
Given such a matching $M$, one can first choose, for each closer immediately followed by an opener, either to merge them into one transitory vertex or to leave them as they are; then one can insert singleton vertices in any position.

Given a matching $M$, let $\val(M)$ denote the number of closers immediately followed by openers in $M$. We call these {\em valleys} of $M$ because, if $\kappa(M)=(R,F)$, they correspond directly to the valleys of $F$, that is,
$\val(M) = \val(F)$.
From the above construction, if $$A(v,z)=\sum_{n\ge 0}\sum_{M\in\M_n(\tau)}u^{\val(M)}z^n$$ is the generating function for $\tau$-avoiding matchings with
respect to the number of valleys, then
$$\tilde{B}(z)=A\left(1+\frac{1}{z},z^2\right)$$ is the generating function for $\tau$-avoiding set partitions without singleton blocks, and
\beq\label{eq:B}B(z)=\frac{1}{1-z}\,\tilde{B}\left(\frac{z}{1-z}\right)=\frac{1}{1-z}\,A\left(\frac{1}{z},\frac{z^2}{(1-z)^2}\right)\eeq is the generating function for all $\tau$-avoiding set partitions.

If two patterns satisfy $\sigma\sim\tau$, then $|\M_F(\sigma)|=|\M_F(\tau)|$ for every Ferrers board $F$, and so the above generating function $A(v,z)$ is the same for $\sigma$-avoiding as for $\tau$-avoiding matchings. It follows that $|\P_n(\sigma)|=|\P_n(\tau)|$ for all $n$ in this case.

\subsection{Numerical data}

In the next three sections we discuss in detail matchings and partitions avoiding a pattern in each one of the three shape-Wilf-equivalence classes of patterns of length~3: $123\sim321\sim213$, $231\sim312$, and $132$.
Some numerical data for
matchings and partitions avoiding each pattern is given in Tables~\ref{tab:matchings} and~\ref{tab:partitions}. The sequences for $|\M_n(231)|$, $|\M_n(132)|$, $|\P_n(231)|$ and $|\P_n(132)|$ did not appear
in the OEIS~\cite{OEIS}. No generating functions for the two sequences involving the pattern $132$ are known.

\begin{table}[ht]
$$\begin{array}{|c|cccccccccc|}
\hline
n &  1 & 2 &3&4&5&6&7&8&9&10\\ 
\hline
|\M_n(231)|  & 1 & 3&14&83&570&4318&35068&299907&2668994&24513578\\ 
\hline
|\M_n(123)| &1& 3& 14& 84& 594& 4719& 40898& 379236& 3711916& 37975756\\ 
\hline
|\M_n(132)| & 1& 3& 14& 84& 595& 4750& 41541& 390566 & 3895957 & 40835749 \\
\hline
\end{array}$$
\caption{The first values of the sequences counting matchings that avoid a pattern of length~3.}
\label{tab:matchings}
\end{table}

\begin{table}[ht]
$$\begin{array}{|c|cccccccccccc|}
\hline
n & 0& 1 & 2 &3&4&5&6&7&8&9&10&11\\
\hline
|\P_n(231)| & 1 & 1 & 2&5&15&52&202&858&3909&18822&94712&493834\\ 
\hline
|\P_n(123)|& 1& 1& 2& 5& 15& 52& 202& 859& 3930& 19095& 97566 & 520257\\ 
\hline
|\P_n(132)| & 1 & 1 & 2 & 5 & 15 & 52 & 202 & 859 & 3930 & 19096 & 97593 & 520694\\
\hline
\end{array}$$
\caption{The first values of the sequences counting set partitions that avoid a pattern of length~3.}
\label{tab:partitions}
\end{table}

\section{The patterns $123\sim321\sim 213$}\label{sec:123}

\subsection{Background}

As mentioned in the introduction, the shape-Wilf-equivalences $123\sim321$ and $123\sim213$ were first proved by Backelin, West and Xin~\cite{BWX} using complicated arguments.
In the context of matchings, the statement $123\sim321\sim 213$ is equivalent to the fact that $|\M_F(123)|=|\M_F(321)|=|\M_F(213)|$ for every Ferrers board $F$.
Recall that $12\dots k$-avoiding (resp. $k\dots 21$-avoiding) matchings are also known as $k$-nonnesting (resp. $k$-noncrossing) matchings.
A shape-preserving bijection between $k$-noncrossing and $k$-nonnesting matchings (which also extends to partitions)
was given by Chen et al~\cite{CDDSY} via vacillating tableaux, providing a simpler proof of the equivalence $12\dots k\sim k\dots 21$.
This bijection was later reformulated by Krattenthaler~\cite{Kra06} in terms of full rook placements and growth diagrams.

Chen et al~\cite{CDDSY} also give a bijection between $k$-noncrossing matchings and certain lattice walks which, in the case $k=3$, can be interpreted as pairs of noncrossing Dyck paths.
It follows that the number of $321$-avoiding matchings on $[2n]$ is the determinant of Catalan numbers $C_nC_{n+2}-C_{n+1}^2$. This formula was first found by Gouyou-Beauchamps~\cite{Gou}
via a recursive bijection from $123$-avoiding matchings and pairs of noncrossing Dyck paths. Thus, the enumeration sequence for $123$-avoiding matchings is P-recursive, which
is equivalent to the corresponding generating function being $D$-finite, but it is not algebraic. A recurrence for the number of $123$-avoiding partitions was found by Bousquel-M\'elou and Xin~\cite{BMXin},
who prove that this sequence is also $D$-finite but not algebraic.

Another proof of the equivalence $123\sim213$ was given by Jel\'{\i}nek~\cite{Jel} by means of a bijection between $213$-avoiding matchings and pairs of noncrossing Dyck paths.
In Section~\ref{sec:bij213} we provide a much simpler bijection between these two sets.

\subsection{A simple bijection between $321$-avoiding and $213$-avoiding matchings}\label{sec:bij213}

The main goal of this section is to provide a simple bijective proof of the fact that $321\sim213$, which is stated in Theorem~\ref{thm:1} below.
For the rest of this section, we fix a Ferrers board $F\in\F_n$, and we let $V_i$ denote the $i$th vertex on the border of $F$.
Recall that $\D_F^2$ denotes the set of pairs of noncrossing Dyck paths where $D_F$ is the top path.

\begin{theorem}\label{thm:1}
There are explicit bijections
$$ \Delta_{321}: \M_F(321) \rightarrow \D_F^2$$
and
$$ \Delta_{213} : \M_F(213) \rightarrow \D_F^2.$$
Therefore, $321\sim 213$.
\end{theorem}

This theorem will follow from Theorems~\ref{thm:delta321} and~\ref{thm:delta213} below.
The bijection $\Delta_{321}$ was first constructed by Chen et al. in~\cite{CDDSY} using vacillating tableaux. Here we provide a short description of this bijection in our language.
Recall that matchings can be viewed as full rook placements via the bijection $\kappa:\M_F\rightarrow\R_F$ described in Section~\ref{sec:matchings}.

It will be convenient to identify a Dyck path $D\in\D_n$ with the sequence $d_0d_1\ldots d_{2n}$ that records the distances from its vertices to the diagonal $y=n-x$
(recall that our Dyck paths start at $(0,n)$ and end at $(n,0)$). More precisely, define $d_0 = 0 = d_{2n}$, and for each $0\le i< 2n$, let $d_{i+1} = d_i+1$ if $V_{i+1}$ is
to the right of $V_i$, and  $d_{i+1}= d_i -1$ if $V_{i+1}$ is below $V_i$. We call $d_0d_1\ldots d_{2n}$ the \emph{height sequence} of $D$.
Fix $h_0h_1\ldots h_{2n}$ to be the height sequence of $D_F$.
Recall that $R\,\cap\,\Gamma(V_i)$ denotes the restriction of $R$ to the rectangle with vertices at the origin and at $V_i$.

\begin{lemma}\label{lem:height_sequence}
For any $(R,F)\in\R_F$ we have
$$
h_i = |R \,\cap\, \Gamma(V_i)|
$$
for all $0\le i\le 2n$.
\end{lemma}

\begin{proof}
For $i=0$ the result is clear since $h_0= 0$ and $\Gamma(V_0)$ is empty. Since $(R,F)$ is a full rook placement, it has a rook in each row and column. It follows that for $0\le i<2n$,
$$h_{i+1} - h_i =   |R \,\cap\, \Gamma(V_{i+1})| - |R \,\cap\, \Gamma(V_i)|,$$
its value being $1$ or $-1$ depending on whether $V_{i+1}$ is to the right of or below $V_i$, respectively.
\end{proof}

For $(R,F)\in \R_F$, define the sequence $j_0\ldots j_{2n}$ by letting
$$j_i = 2\ell_i - h_i,$$
where $\ell_i$ is the length of the longest increasing sequence in $R\,\cap\,\Gamma(V_i)$. The following property of
this sequence will be used to define $\Delta_{321}$.

\begin{lemma}\label{lem:2}
If $(R,F)\in \R_F(321)$, then the sequence $ j_0\ldots j_{2n}$ is the height sequence of a Dyck path $D_{R,F}$, and it satisfies $j_i\leq h_i$ for all $i$. Thus, $(D_{R,F},D_F)\in \D^2_F$.
\end{lemma}

\begin{proof}
Clearly $j_0 = 0 = j_{2n}$, and $j_{i+1}=j_i \pm 1$ for $0\le i<2n$.  It remains to show that $0\le j_i\le h_i$ for all $i$. The permutation determined by $R\,\cap\,\Gamma(V_i)$ is $321$-avoiding, so it is the union of
two disjoint increasing sequences. Letting $r$ and $s$ be the lengths of these sequences, with $r\ge s$, we have
$$2h_i\geq 2\ell_i\geq r + s = h_i,$$
by Lemma~\ref{lem:height_sequence}. It follows that $0\le 2\ell_i-h_i\le h_i$.
\end{proof}

Define a map $\delta_{321}:\R_F(321)\rightarrow \D_F^2$ by letting $\delta_{321}(R,F)=(D_{R,F},D_F)$, with $D_{R,F}$ as defined in Lemma~\ref{lem:2}. Finally, let $\Delta_{321}=\delta_{321}\circ \kappa$.
An example of the map $\delta_{321}$ is given in Figure~\ref{fig:del_321}. It remains to show that $\delta_{321}$ is bijective.

\begin{figure}[h!]
    \def\R{-- ++(1,0) [fill] circle(1.3pt)}
    \def\D{-- ++(0,-1) [fill] circle(1.3pt)}
\centering

\begin{tikzpicture}[scale=0.4]
\begin{scope}[shift={(0,0)}, scale = 1.3]
\foreach \i / \h in {0/6,1/6, 2/6, 3/6, 4/5, 5/3, 6/3}
	{
		\draw (\i, 0) -- (\i, \h);
	}
\foreach \i / \w in {0/6,1/6, 2/6, 3/6, 4/4, 5/4, 6/3}
	{
		\draw (0, \i) -- (\w, \i);
	}
\foreach \i / \w in {0.5/4.5,1.5/0.5, 2.5/5.5, 3.5/3.5, 4.5/2.5, 5.5/1.5}
	{
		\draw (\i, \w) node {{$\times$}};
	}

\draw (0,6.3) node {\footnotesize 0};
\draw (1,6.3) node {\footnotesize 1};
\draw (2,6.3) node {\footnotesize 0};
\draw (3,6.3) node {\footnotesize 1};

\draw (3.3,5.3) node {\footnotesize 0};
\draw (4,   5.3) node {\footnotesize 1};

\draw (4.3,4) node {\footnotesize 2};

\draw (4.3,3.3) node {\footnotesize 1};
\draw (5,3.3) node {\footnotesize 2};
\draw (6,3.3) node {\footnotesize 1};

\draw (6.3,2) node {\footnotesize 2};
\draw (6.3,1) node {\footnotesize 1};
\draw (6.3,0) node {\footnotesize 0};

\draw[->] (8,4) -- (12,4);
\draw (10,4.7) node {$\delta_{321}$};
\end{scope}

\begin{scope}[shift = {(16,8)}, scale=1.5]
\draw[line width = .4mm] (0.1,0.1) circle(1.3pt)\R\R\R\D\R\D\D\R\R\D\D\D;
\draw[line width = .4mm] (0,0) circle(1.3pt)\R\D\R\D\R\R\D\R\D\R\D\D;
\end{scope}
\end{tikzpicture}

\caption{An example of the bijection $\delta_{321}$. The sequence $j_0\ldots j_{2n}$ is written along the border of~$F$.}
\label{fig:del_321}
\end{figure}
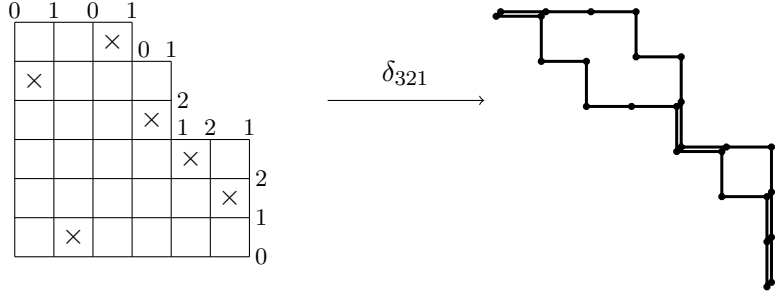

\begin{theorem}\label{thm:delta321}
The map $\delta_{321}:\R_F(321)\rightarrow \D_F^2$ is a bijection, and thus so is $\Delta_{321}:\M_F(321) \to \D_F^2$.
\end{theorem}

\begin{proof}
The statement is easier to prove from an alternative description of $\delta_{321}$ in terms of Fomin's growth diagrams~\cite{Fom86}.
Define an \emph{$F$-oscillating tableau} to be a sequence of integer partitions $\set{\lambda^i}_{0\leq i\leq 2n}$ such that $\lambda^0 = \emptyset = \lambda^{2n}$ and, for $0\le i<2n$,
$\lambda^{i+1}$ is obtained from $\lambda^{i}$ by adding (removing) a square if $V_{i+1}$ is to the right of (respectively, below) $V_i$.
Denote by $\Lambda_F$ the set of $F$-oscillating tableaux. Fomin's growth diagram algorithm~\cite[Chapter 7]{EC2} gives a bijection between $\R_F$ and $\Lambda_F$ with the following property:
if $(R,F)\in\R_F$ is mapped to $\{\lambda^i\}_i\in\Lambda_F$, then $\lambda^i$ is the shape of the recording (and insertion) tableaux under the RSK correspondence of the permutation determined by $R\,\cap\,\Gamma(V_i)$.
It is a well-known result of Schensted~\cite{Sch} that the length of the longest increasing sequence in this permutation is then the size of the largest part of $\lambda^i$, which we denote $\rho(\lambda^i)$.
Similarly, the length of its longest decreasing sequence is the number of parts of $\lambda^i$.
Thus, Fomin's growth diagram algorithm restricts to a bijection between
$\R_F(321)$ and the subset $\Lambda_F(321)\subset \Lambda_F$ consisting of $F$-oscillating tableaux where each partition has at most two parts.

Given $(R,F)\in\R_F(321)$, we can now define its image $\delta_{321}(R,F)=(D_{R,F},D_F)$ in terms of the corresponding $F$-oscillating tableau $\{\lambda^i\}_i\in\Lambda_F(321)$. Indeed,
$$|\lambda^i| =|R\,\cap\,\Gamma(V_i)| =  h_i \qquad\textrm{and}\qquad  2\rho(\lambda^i)-|\lambda^i| = 2\ell_i - h_i = j_i .$$

Proving that $\delta_{321}:\R_F(321)\rightarrow \D_F^2$ is a bijection is therefore equivalent to showing that the map $\{\lambda^i\}_i\mapsto (D_{R,F},D_F)$
is a bijection from $\Lambda_F(321)$ to $\D^2_F$.
Injectivity is clear since we may recover each $\lambda^i$ from $h_i=|\lambda^i|$ and $j_i = 2\rho(\lambda^i)-|\lambda^i|$ by observing that $\rho(\lambda^i)=(h_i+j_i)/2$ is the largest part of $\lambda^i$,
and $h_i - \rho(\lambda^i)$ is the other part.
To show surjectivity, note that any $(D_0,D_F)\in D^2_F$, where $k_0\dots k_{2n}$ is the height sequence for $D_0$, is the image of the $F$-oscillating tableaux $\{\lambda^i\}_i\in\Lambda_F(321)$, where
$$\lambda^i = \left(\frac{h_i+k_i}{2},\frac{h_i - k_i}{2}\right).$$
\end{proof}

Let us mention an alternative simple way to describe the path $D_{R,F}$ in the definition of $\delta_{321}$:
first label the vertices in $D_F$ by letting vertex $V_i$ have label $\ell_i$ (the length of the longest increasing sequence in $R\,\cap\,\Gamma(V_i)$);
then take $D_F$ and switch (i.e., $N$ becomes $E$ and viceversa) the steps that have the same label at both endpoints. The resulting path is $D_{R,F}$.

\ms

Now we turn to the second part of the proof of Theorem~\ref{thm:1}.
Even though a different bijection between $\M_F(213)$ and $\D_F^2$ has already been given by Jel\'inek in~\cite{Jel}, here we present a much simpler bijection $\Delta_{213}$ through a short pictorial argument.

As in the case of $321$-avoiding matchings, it is convenient to let $\Delta_{213}=\delta_{213}\circ \kappa$, where the map $\delta_{213}:\R_F(213)\rightarrow \D_F^2$ is defined as $\delta_{213}(R,F) = (D, D_F)$, where $D$ is given by the following construction. We use a bijection due to Krattenthaler~\cite{Kra01} between $213$-avoiding permutations and Dyck paths.
As the pattern $213$ ends with its largest entry, the fact that $(R,F)$ is $213$-avoiding implies that the permutation $\pi_R$ is in fact an element of $\S_n(213)$.
Let $F_R$ be the minimal Ferrers board that contains $R$. Krattenthaler's bijection is the map sending $\pi_R\in \S_n(213)$ to the border of $F_R$. The placement $R$ can be recovered from $F_R$
by the following iterative procedure: begin by placing a rook in the rightmost column of the top row of $F_R$; having placed rooks in the top $k$ rows,
the rook in the $(k+1)$st row from the top is placed in the rightmost column which does not already contain a rook. We define the second path in $\delta_{213}(R,F)$ to be $D=D_{F_R}$.
Note that $F_R\subseteq F$ by definition, so $D_{F_R}$ and $D_F$ are noncrossing Dyck paths.
The following theorem is now clear.

\begin{theorem}\label{thm:delta213}
The map $\delta_{213} : \R_F(213) \to \D_F^2$ is a bijection, and thus so is $\Delta_{213}:\M_F(213) \to \D^2_F$.
\end{theorem}

An example of the map $\delta_{213}$, together with the complete bijection from between $\M_{F}(321)$ and $\M_F(213)$, is given in Figure~\ref{fig:4}.

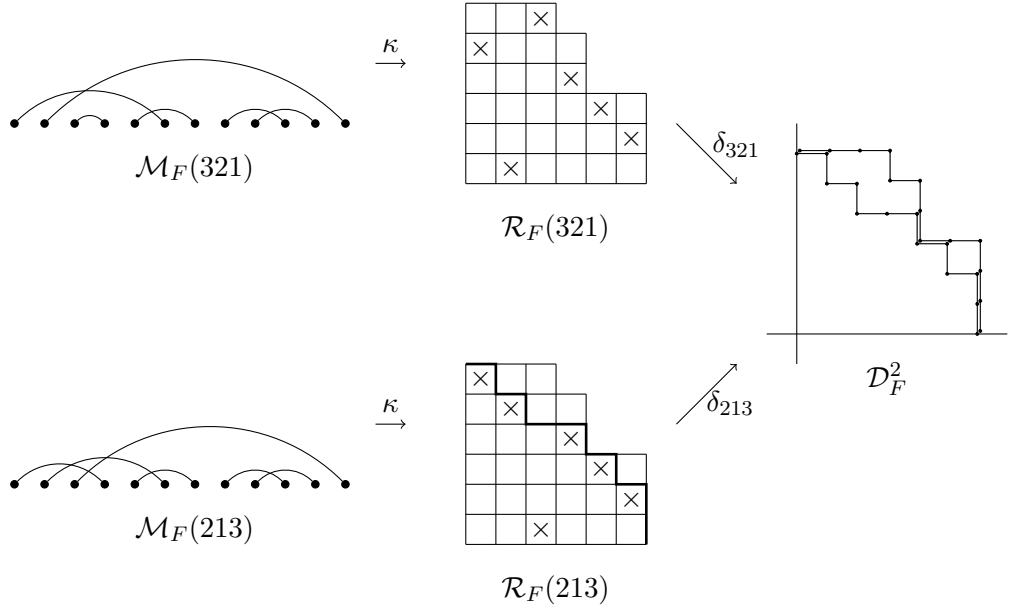
\begin{figure}[h!]
\begin{center}
    \def\U{-- ++(1,0) [fill] circle(1.3pt)}
    \def\N{-- ++(0,-1) [fill] circle(1.3pt)}
\begin{tikzpicture}[scale=0.4]

\begin{scope}[shift={(0,11)}]
	\foreach \i in {0,...,11}
	{
		\node[pnt] at (\i,0)(\i){};
	}
	\draw(0)  to [bend left=45] (5);
	\draw(1)  to [bend left=45] (11);
	\draw(2)  to [bend left=45] (3);
	\draw(4)  to [bend left=45] (6);
	\draw(7)  to [bend left=45] (9);
	\draw(8)  to [bend left=45] (10);
	\draw[->] (12,2) -- (13,2);
	\draw (12.5,2.6) node {$\kappa$};
	
	\draw (6,-1.5) node {$\M_F(321)$};
\end{scope}

\begin{scope}[shift={(15,9)}]
\foreach \i / \h in {0/6,1/6, 2/6, 3/6, 4/5, 5/3, 6/3}
	{
		\draw (\i, 0) -- (\i, \h);
	}
\foreach \i / \w in {0/6,1/6, 2/6, 3/6, 4/4, 5/4, 6/3}
	{
		\draw (0, \i) -- (\w, \i);
	}
\foreach \i / \w in {0.5/4.5,1.5/0.5, 2.5/5.5, 3.5/3.5, 4.5/2.5, 5.5/1.5}
	{
		\draw (\i, \w) node {{$\times$}};
	}
\path[->] (7,2) edge  (9,0);
\draw (9,1.3) node {$\delta_{321}$};

\draw (3,-1.5) node {$\R_F(321)$};
\end{scope}

\begin{scope}[shift = {(26,10)}]
\draw (0.1,0.1) circle(1.3pt) \U\U\U\N\U\N\N\U\U\N\N\N;
\draw (0,0) circle(1.3pt) \U\N\U\N\U\U\N\U\N\U\N\N;
\draw (0,1) -- (0,-7);
\draw (-1,-6) -- (7,-6);

\draw (3,-7.5) node {$\D^2_F$};
\end{scope}

\begin{scope}[shift={(0,-1)}]
	\foreach \i in {0,...,11}
	{
		\node[pnt] at (\i,0)(\i){};
	}
	\draw(0)  to [bend left=45] (3);
	\draw(1)  to [bend left=45] (5);
	\draw(2)  to [bend left=45] (11);
	\draw(4)  to [bend left=45] (6);
	\draw(7)  to [bend left=45] (9);
	\draw(8)  to [bend left=45] (10);
	\draw[->] (12,2) -- (13,2);
	\draw (12.5,2.6) node {$\kappa$};
	\draw (6,-1.5) node {$\M_F(213)$};
\end{scope}

\begin{scope}[shift = {(15,-3)}]
\foreach \i / \h in {0/6,1/6, 2/6, 3/6, 4/5, 5/3, 6/3}
	{
		\draw (\i, 0) -- (\i, \h);
	}
\foreach \i / \w in {0/6,1/6, 2/6, 3/6, 4/4, 5/4, 6/3}
	{
		\draw (0, \i) -- (\w, \i);
	}
\foreach \i / \w in {0.5/5.5,1.5/4.5, 2.5/0.5, 3.5/3.5, 4.5/2.5, 5.5/1.5}
	{
		\draw (\i, \w) node {{$\times$}};
	}
\draw[line width = .35mm] (0,6) -- (1,6) -- (1,5) -- (2,5) -- (2,4) -- (4,4) -- (4,3) -- (4,3) -- (5,3) -- (5,2) -- (6,2) -- (6,0);
\path[->] (7,4) edge (9,6);
\draw (8.8,4.7) node {$\delta_{213}$};
\draw (3,-1.5) node {$\R_F(213)$};
\end{scope}

\end{tikzpicture}
\caption{An example of the bijection between $\M_{F}(321)$ and $\M_F(213)$. The bold path on the bottom Ferrers board represents the border of $F_R$.} 
\label{fig:4}
\end{center}
\end{figure}

It is worth mentioning that in the particular case that $F\in\F_n$ is the square board, the composition $\Delta_{213}^{-1}\circ\Delta_{132}$ gives a bijection between $\S_n(321)$ and $\S_n(213)$ which
coincides, up to symmetry, with a bijection of Elizalde and Pak~\cite{EliPak}, that was used to prove that the number of fixed points
and the number of excedances have the same distribution on both sets. More precisely, the image of $\pi\in\S_n(321)$ by their bijection $\Theta$ is the permutation obtained by reflecting
$\Delta_{213}^{-1}(\Delta_{132}(\pi))$ (viewed as a rook placement) over the diagonal $y=n-x$. Our bijection can thus be interpreted as a generalization of $\Theta$ to arbitrary Ferrers boards.

\subsection{$321$- and $213$-avoiding matchings with fixed points}\label{sec:Gou}

In \cite{Gou} Gouyou-Beauchamps found exact formulas for the number of standard Young tableaux having $n$ squares and at most $p$ rows, for $p\in\{4,5\}$.  Additionally, he gave a formula for the number of Young tableaux having $n$ squares and exactly $k$ columns of odd height. The key theorem in~\cite{Gou} is a bijection between involutions of length $2n+k$ with $k$ fixed points that avoid a decreasing sequence of length 5 (i.e., the pattern $54321$) and pairs of noncrossing Dyck paths of semilength $n+k$ that end with $k$ down steps.  To interpret this bijection in our language we will consider, in this section, matchings that are not necessarily perfect, that is, they may have unmatched elements. Recall the canonical correspondence between involutions of $[2n]$ with no fixed points and perfect matchings in $\M_n$, obtained by matching pairs of elements that belong to the same 2-cycle in the involution. One can extend this correspondence by allowing fixed points in the involution, which become vertices of degree 0 (which we also call \emph{fixed points}) in the matching. Under this correspondence, $54321$-avoiding involutions of $[2n+k]$ with $k$ fixed points are mapped to matchings on $[2n+k]$ with $k$ fixed points satisfying:
\bit \item if we remove the fixed points (and relabel the remaining elements increasingly with $1,2,\dots,2n$), we get an element of $\M_n(123)$, and
\item for any 5 vertices $x_1<\dots<x_5$ where the pairs $(x_1,x_5)$ and $(x_2,x_4)$ are matched, $x_3$ is not a fixed point.
\eit
Let $\M_n^k(123)$ denote the set of such matchings, and let $\D_{n,k}^2$ denote  the set of all pairs $(D_0,D_1)\in\D^2_{n+k}$ where $D_0$ and $D_1$ end with $k$ south steps.
Gouyou-Beauchamps' bijection \cite{Gou} can thus be interpreted as a bijection between $\M_n^k(123)$ and $\D_{n,k}^2$.

In this section we construct two related but simpler bijections between pattern-avoiding matchings with fixed points and $\D_{n,k}^2$.
Let $\M_n^k(213)$ (respectively,  $\M_n^k(321)$) be the set of matchings on $[2n+k]$ with $k$ fixed points with the property that removing the fixed points produces an element of $\M_n(213)$ (respectively, $\M_n(321)$), and for any 5 vertices $x_1<\dots<x_5$ where the pairs $(x_1,x_5)$ and $(x_3,x_4)$ are matched, $x_2$ is not a fixed point (respectively, $(x_1,x_4)$ and $(x_2,x_5)$ are matched, $x_3$ is not a fixed point).
The bijections in the following theorem generalize those in Theorems~\ref{thm:delta321} and~\ref{thm:delta213}.

\begin{theorem}
There are explicit bijections between $\M_n^k(321)$ and $\D_{n,k}^2$, and between $\M_n^k(213)$ and $\D_{n,k}^2$.
\end{theorem}

\begin{proof}
In this proof we use $\tau$ to denote either of the patterns $321$ and $213$.
 Let $\R_n^k$ be the set of pairs $(R,F)\in\R_{n+k}$ such that the bottom $k$ rows of $F$ have length $n+k$, and the placement $R$ restricted to these rows is an increasing sequence.
Let $\R_n^k(\tau) = \R_n^k \,\cap\, \R_{n+k}(\tau)$.

We begin by establishing a bijection $\kappa'$ between $\M_n^k(\tau)$ and $\R_n^k(\tau)$. Given $M\in\M_n^k(\tau)$, add $k$ new vertices $2n+k+1, \ldots, 2n+2k$ and $k$ new edges $(x_i, 2n+2k+1-i)$, where
$x_1< \dots< x_k$ are the fixed points of $M$. Denote the resulting perfect matching by $M^+$.
The map $M \to M^+$ is a bijection from  $\M_n^k(\tau)$ to the set of matchings in $\M_{n+k}(\tau)$ with the property that the vertices $2n+k,\ldots,2n+2k$ are closers whose edges do not cross.
Define $\kappa'(M)$ to be $\kappa(M^+)$, where $\kappa$ is defined in Section~\ref{sec:matchings}. It is clear from the construction that $\kappa'(M)\in\R_n^k(\tau)$ and that $\kappa'$ is a bijection.

Recall from Section~\ref{sec:bij213} that for a fixed Ferrers board $F$, the map $\delta_\tau$ is a bijection between $\R_F(\tau)$ and $\D^2_F$.
Considering the disjoint unions $$\R_{n+k}(\tau)=\bigcup_{F\in\F_{n+k}} \R_F(\tau)  \quad\mathrm{and}\quad \D^2_{n+k}=\bigcup_{F\in\F_{n+k}} \D^2_F ,$$ it follows
that $\delta_\tau$ gives a bijection between $\R_{n+k}(\tau)$ and $\D^2_{n+k}$, which we also call $\delta_\tau$.
We now prove that $\delta_\tau$ restricts to a bijection between $\R_n^k(\tau)$ and $\D^2_{n,k}$.
Suppose that $(R,F)\in\R_{n+k}(\tau)$ and $\delta_\tau((R,F))=(D_0,D_1)$. Clearly, the bottom $k$ rows of $F$ have length $n+k$ if and only if $D_1$ ends with $k$ south steps.
It remains to show that, in this case, the restriction of $R$ to the bottom $k$ rows of $F$ forms an increasing sequence (call this the {\em $k$-increasing property}) if and only if the path $D_1$ also ends with $k$ south steps.
In the case $\tau = 213$, it follows from the description of the map $R\mapsto F_R$ and its inverse (in the paragraph preceding Theorem~\ref{thm:delta213})
that the $k$-increasing property is equivalent to the fact that the bottom $k$ rows of the minimal Ferrers board $F_R$ have length $n+k$.
In the case $\tau = 321$, the $k$-increasing property guarantees that the sequence $J(R)$ defined in Section~\ref{sec:bij213} ends with $k(k-1)\ldots 0$,
and the converse is also clear given that the bottom $k$ rows of $F$ have length $n+k$.

The composition $\delta_\tau\circ\kappa'$ is a bijection between $\M_n^k(\tau)$ and $\D_{n,k}^2$. An example of the bijection $\delta_{321}\circ\kappa'$ is given in Figure~\ref{fig:5}.
\end{proof}

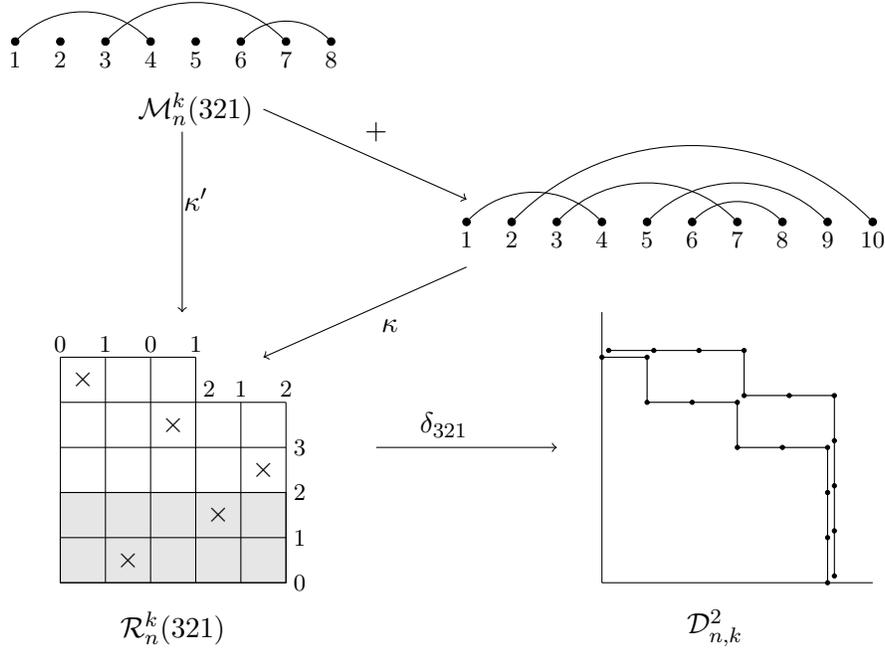
\begin{figure}[ht]
\begin{center}
\begin{tikzpicture}[scale=0.6]

\begin{scope}[shift ={(0,5)}]
	\foreach \i in {0,...,7}
	{
		\node[pnt] at (\i,0)(\i){};
	}
	\foreach \i / \j in {0/1,1/2,2/3, 3/4, 4/5, 5/6, 6/7, 7/8}
	{
		\draw (\i) node[below] {\footnotesize \j};
	}
	
	\draw(0)  to [bend left=45] (3);
	\draw(2)  to [bend left=45] (6);
	\draw(5)  to [bend left=45] (7);

    \node at (4, -1.5) {$\M_n^k(321)$};
    \draw[->] (3.7,-2) -- (3.7,-6);
    \node at (4, -3.5) {$\kappa'$};

    \draw[->] (5.5,-1.5) -- (10,-3.5);
        \node at (8, -2) {$+$};

\end{scope}

\begin{scope}[shift = {(10,1)}]
	\foreach \i in {0,...,9}
	{
		\node[pnt] at (\i,0)(\i){};
	}
	\foreach \i / \j in {0/1,1/2,2/3, 3/4, 4/5, 5/6, 6/7, 7/8,8/9,9/10}
	{
		\draw (\i) node[below] {\footnotesize \j};
	}	
	
	\draw(0)  to [bend left=45] (3);
	\draw(1)  to [bend left=45] (9);
	\draw(2)  to [bend left=45] (6);
	\draw(4)  to [bend left=45] (8);
	\draw(5)  to [bend left=45] (7);
	
	\draw[->] (0,-1) -- (-4.5,-3);
    	\node at (-1.7, -2.3) {$\kappa$};

\end{scope}

\begin{scope}[shift = {(1,-7)}]

	\draw[fill=gray!20] (0,0) rectangle (5,2);
	
	\foreach \i / \c / \r in {0/5/5,1/5/5,2/5/5, 3/5/5, 4/4/5, 5/4/3}
	{
		\draw (\i,0) -- (\i,\c);
		\draw (0,\i) -- (\r,\i);
	}

	\node at (0.5,4.5) {$\times$};
	\node at (1.5,0.5) {$\times$};
	\node at (2.5,3.5) {$\times$};
	\node at (3.5,1.5) {$\times$};
	\node at (4.5,2.5) {$\times$};
	
	\node at (0,5.3) {\footnotesize 0};
	\node at (1,5.3) {\footnotesize 1};
	\node at (2,5.3) {\footnotesize 0};
	\node at (3,5.3) {\footnotesize 1};
	
	\node at (3.3,4.3) {\footnotesize 2};
	\node at (4,   4.3) {\footnotesize 1};
	\node at (5,   4.3) {\footnotesize 2};
	
	\node at (5.3,3) {\footnotesize 3};
	\node at (5.3,2) {\footnotesize 2};
	\node at (5.3,1) {\footnotesize 1};
	\node at (5.3,0) {\footnotesize 0};
	
    	\draw[->] (7,3) -- (11,3);
	\node at (8.5, 3.5) {$\delta_{321}$};
    \node at (2.5, -1) {$\R_n^k(321)$};

\end{scope}

\begin{scope}[shift = {(13,-2)}]
    \node at (2.5, -6) {$\D_{n,k}^2$};
	\def\U{-- ++(1,0) [fill] circle(1.3pt)}
	\def\D{-- ++(0,-1) [fill] circle(1.3pt)}
	\draw (0,-5) -- (6,-5);
	\draw (0,-5) -- (0,1);
	\draw (0.15,0.15) circle(1.3pt) \U\U\U\D\U\U\D\D\D\D;
	\draw (0,0) circle(1.3pt) \U\D\U\U\D\U\U\D\D\D;
	
\end{scope}
\end{tikzpicture}
\caption{An example of the bijection $\delta_{321}\circ\kappa':\M_n^k(321)\rightarrow\D^2_{n,k}$. The shaded rows in the Ferrers board correspond to the two vertices added to the matching.}
\label{fig:5}
\end{center}
\end{figure}

\section{The patterns $231\sim312$}\label{sec:231}

\subsection{Background}\label{sec:231background}

The first proof of the fact that the patterns $231$ and $312$ are shape-Wilf-equivalent was given by Stankova and West~\cite{SW}. Later, Bloom and Saracino~\cite{BloSar} gave a more direct proof, constructing a bijection between $231$-avoiding and $312$-avoiding full rook placements of any given Ferrers board.
As shown in Figure~\ref{fig:arcpatterns}, the pattern $231$ (resp. $312$) in matchings and partitions consists of
two nested arcs crossed by a mutual arc, ending to the right (resp. left) of the nested arcs. In terms of matchings,
Bloom and Saracino's map translates into a bijection between $231$-avoiding and $312$-avoiding matchings that preserves the sequence of openers and closers.

The main ingredient in Bloom and Saracino's construction is a bijection between $231$-avoiding full rook placements of a given Ferrers board $F\in\F_n$ and certain labelings of the vertices on the border of $F$.
Recall that the vertices $V_0V_1\ldots V_{2n}$ are ordered from $(0,n)$ to $(n,0)$.

We define a \emph{labeled} Dyck path of semilength $n$ to be a pair $(D,\alpha)$ where $D\in\D_n$
and $\alpha = \alpha_0\alpha_1\ldots\alpha_{2n}$ is an integer sequence with the following monotonicity property: if $V_i$ is to the left of $V_{i+1}$ then $\alpha_i \leq \alpha_{i+1} \leq \alpha_i +1$, else $\alpha_i\geq \alpha_{i+1} \geq \alpha_i-1$. We think of $\alpha_i$ as the label of vertex $V_i$.

We say that two vertices $V_i=(x_i,y_i)$ and $V_j=(x_j,y_j)$ of $D$ are {\em aligned} if $x_i-x_j=y_i -y_j$ and the line segment connecting the points $V_i$ and $V_j$ lies strictly below $D$ (except for the endpoints of the segment, which are on $D$). We say that a labeled Dyck path $(D,\alpha)$ has the \emph{diagonal property} if for any two aligned vertices $V_i$ and $V_j$ with $i<j$, we have $\alpha_i\geq \alpha_j$.
We say $(D,\alpha)$ satisfies the \emph{0-condition} if $\alpha_0 = 0 = \alpha_{2n}$ and $\alpha$ contains no consecutive zeros (equivalently, for each $i$ one has $\alpha_i=0$ if and only if $V_i$ lies on the diagonal $y=n-x$).
For $F\in \F_n$, we denote by $\L_F$ the set of labelings $(D_F,\alpha)$ of the boundary of $F$ that satisfy both the diagonal property and the 0-condition.  We also let $\L_n = \bigcup_{F\in \F_n} \L_F$.

Bloom and Saracino's bijection~\cite{BloSar} between placements and labeled Dyck paths is the map $\Pi:\R_F(312)\to \L_F$ that sends $(R,F)\in\R_F(312)$ to the pair $(D_F,\alpha)$ where, for $0\le i\le 2n$, the label $\alpha_i$ is the length of the longest increasing sequence in $R\,\cap\, \Gamma(V_i)$. In a slight abuse of notation, we also denote by $\Pi$ the bijection induced by $\Pi$ from $\R_n(312) = \bigcup_{F\in \F_n} \R_F(312)$ to $\L_n = \bigcup_{F\in F_n} \L_F$.

\subsection{$312$-avoiding matchings}\label{sec:312matchings}

In this section we enumerate $312$-avoiding matchings, or equivalently, $231$-avoiding ones. 

\begin{theorem}\label{thm:312matchings}
The generating function for $312$-avoiding matchings is $$\sum_{n\ge 0} |\M_n(312)|z^n=\frac{54z}{1+36z-(1-12z)^{3/2}}.$$
The asymptotic behavior of is coefficients is given by
\beq\label{eq:asym312mat}|\M_n(312)|\sim \frac{3^3}{2^5\sqrt{\pi n^5}}\,12^n.\eeq
\end{theorem}

\begin{proof}
We first translate the problem into an enumeration of labeled Dyck paths.
The composition $\Pi\circ\kappa$ is a bijection between $\M_n(312)$ and $\L_n$, so we have
$$L(z) =\sum_{n\ge 0} |\M_n(312)|z^n=\sum_{n\ge 0}|\L_n|z^n.$$
We will find an expression for $L(z)$ using the recursive structure of Dyck paths: every $D\in\D_n$ with $n\ge1$ uniquely decomposes as $eD_1sD_2$ where $e$ is an east step, $s$ is a south step, and $D_1$ and $D_2$ are Dyck paths. Even though this decomposition can be extended to deal with labeled Dyck paths by transferring the label on each vertex of $D$ to the corresponding vertex of $eD_1s$ or $D_2$,
the fact that the labels on $eD_1s$ satisfy the 0-condition does not guarantee that the labels on $D_1$ do,
even if their values are decreased by 1.

To deal with this problem, we relax the 0-condition and consider the larger set $\K_n$ consisting of all labeled Dyck paths $(D,\alpha)$ of semilength $n$ that have the diagonal property and satisfy $\alpha_{2n}=0$. Let
$\K=\bigcup_{n\ge 0}\K_n$, and denote by
$$K(u,z)=\sum_{0 \leq n} \sum_{(D,\alpha)\in\K_n}u^{\alpha_0}z^n$$ the generating function for such paths according to the value of the first label.

To obtain an equation for $K(u,z)$, first consider the following operation: given $(A,\alpha)\in \K_i$,
$(B,\beta)\in \K_j$, let $(A,\alpha)\oplus (B,\beta)\in\K_{i+j}$ be the concatenation of Dyck paths $AB$ with labels $(\alpha_0+\beta_0)(\alpha_1+\beta_0)\ldots(\alpha_{2i}+\beta_0)\beta_1\ldots\beta_{2j}$.
In other words, the labels along $A$ are increased by $\beta_0$, and the labels along $B$ do not change.
Every nonempty $(D,\gamma)\in\K$ can be decomposed uniquely as $(D,\gamma)=(eD_1s,\alpha)\oplus (D_2, \beta)$ where $(eD_1s,\alpha),(D_2, \beta)\in \K$. Whereas $(D_2, \beta)$ is an arbitrary element of $\K$,
the labeling $\alpha$ of the elevated Dyck path $eD_1s$ can be of four different types,
according to whether $\alpha_0=\alpha_1$ and whether $\alpha_{2i-1}=\alpha_{2i}$, where $i$ is the semilength of
$eD_1s$. The four possibilities are shown in Figure~\ref{fig:DecompositionOfLabledPath}. In cases (a) and (b), $(D_1,(\alpha_1-1)\ldots(\alpha_{2i-1}-1))$ is an arbitrary element of $\K_{i-1}$. In case (c), so is $(D_1,\alpha_1\ldots\alpha_{2i-1})$. In case (d), $(D_1,\alpha_1\ldots\alpha_{2i-1})$ is an element of $\K_{i-1}$ with $\alpha_1\ge1$.

\begin{figure}[ht]
\begin{center}
\begin{tikzpicture}[scale=0.4]
\def\U{-- ++(1,0) [fill] circle(1.3pt)}
\def\D{-- ++(0,-1) [fill] circle(1.3pt)}
\begin{scope}[shift = {(0,0)}]
	\draw(0,4) --(4,0);	
	\draw (0,4) circle(1.3pt) \U\U;
	\draw (4,2) circle(1.3pt) \D\D;
	\node at (3.5, 3.5) {$\ddots$};
	\draw [thick,decoration={brace,mirror},decorate] (4.6,1) -- (4.6,4);
	\node at (5.5,2.5) {$D_1$};
	
	\node at (-.2,4.4) {\footnotesize $a-1$};
	\node at (1,4.4) {\footnotesize $a$};
	
	\node at (4.3,0) {\footnotesize 0};
	\node at (4.3,1) {\footnotesize 1};
	\node at (3,-1.5) {(a)};
\end{scope}
\begin{scope}[shift={(9,0)}]	
	\draw(0,4) --(4,0);
	\draw (0,4) circle(1.3pt) \U\U;
	\draw (4,2) circle(1.3pt) \D\D;
	\node at (3.5, 3.5) {$\ddots$};
	\draw [thick,decoration={brace,mirror},decorate] (4.6,1) -- (4.6,4);
	\node at (5.5,2.5) {$D_1$};
	
	\node at (-.2,4.4) {\footnotesize $a$};
	\node at (1,4.4) {\footnotesize $a$};
	
	\node at (4.3,0) {\footnotesize 0};
	\node at (4.3,1) {\footnotesize 1};
	\node at (3,-1.5) {(b)};
\end{scope}
\begin{scope}[shift = {(18,0)}]
	\draw(0,4) --(4,0);
	\def\U{-- ++(1,0) [fill] circle(1.3pt)}
	\def\D{-- ++(0,-1) [fill] circle(1.3pt)}
	\draw (0,4) circle(1.3pt) \U\U;
	\draw (4,2) circle(1.3pt) \D\D;
	\node at (3.5, 3.5) {$\ddots$};
	\draw [thick,decoration={brace,mirror},decorate] (4.6,1) -- (4.6,4);
	\node at (5.5,2.5) {$D_1$};
	
	\node at (-.2,4.4) {\footnotesize $a$};
	\node at (1,4.4) {\footnotesize $a$};
	
	\node at (4.3,0) {\footnotesize 0};
	\node at (4.3,1) {\footnotesize 0};
	\node at (3,-1.5) {(c)};
\end{scope}
\begin{scope}[shift = {(27,0)}]
	\draw(0,4) --(4,0);
	\def\U{-- ++(1,0) [fill] circle(1.3pt)}
	\def\D{-- ++(0,-1) [fill] circle(1.3pt)}
	\draw (0,4) circle(1.3pt) \U\U;
	\draw (4,2) circle(1.3pt) \D\D;
	\node at (3.5, 3.5) {$\ddots$};
	\draw [thick,decoration={brace,mirror},decorate] (4.6,1) -- (4.6,4);
	\node at (5.5,2.5) {$D_1$};
	
	\node at (-.2,4.4) {\footnotesize $a-1$};
	\node at (1,4.4) {\footnotesize $a$};
	
	\node at (4.3,0) {\footnotesize 0};
	\node at (4.3,1) {\footnotesize 0};
	\node at (3,-1.5) {(d)};
\end{scope}
\end{tikzpicture}
\caption{The four possible labelings of the path $eD_1s$.}
\label{fig:DecompositionOfLabledPath}
\end{center}
\end{figure}
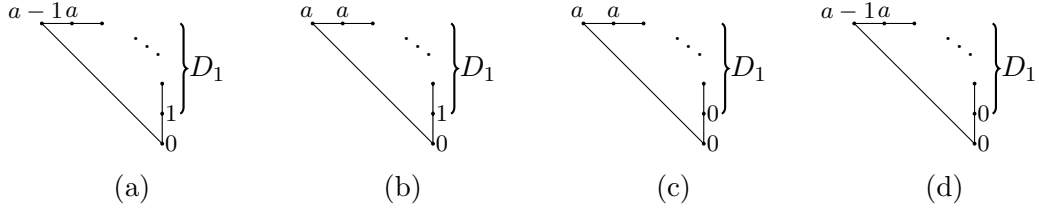

This decomposition translates into the functional equation
\beq\label{eq:Ll} K(u,z)=1+zK(u,z)\left(2K(u,z)+uK(u,z)+\frac{K(u,z)-K(0,z)}{u}\right).\eeq
To solve this equation we use the quadratic method, due to Tutte, as described in~\cite[p.~515]{FS}.
Completing the square in~\eqref{eq:Ll} we get
\beq\label{eq:Ll_square}4(u+1)^4z^2\left(K(u,z)-\frac{u+zK(0,z)}{2(u+1)^2z}\right)^2=(u+zK(0,z))^2-4u(u+1)^2z.\eeq
Let $g(u,z,K(0,z))$ be the right hand side of Eq.~\eqref{eq:Ll_square}.
Let $u(z)$ be a (yet unknown) function of $z$ such that, when substituted for $u$ in the left hand side of~\eqref{eq:Ll_square}, it vanishes, that is,
$$2(u(z)+1)^2zK(u(z),z)-u(z)-zK(0,z)=0.$$ Now, since the left hand side has a double root in $u$, so does the right hand side, so
$g(u(z),z,K(0,z))$ and
$\frac{\partial g}{\partial u}(u(z),z,K(0,z))=0$. Solving this system of two equations for the two unknowns $u(z)$ and $K(0,z)$, we obtain
\beq\label{eq:Ll0z} K(0,z)=-\frac{1-18z-(1-12z)^{3/2}}{54z^2}.\eeq
Substituting in~\eqref{eq:Ll}, we can derive an expression for $K(u,z)$, which is algebraic of degree~4.

Finally, to find $L(z)$, observe that for any $(D,\alpha)\in\L_n$, $D$ can be decomposed uniquely as $D=eA_1seA_2s\dots$, where each $A_j$ is a Dyck path, and if we let $\alpha^{(j)}$ be its sequence of labels decreased by one, then $(A_j,\alpha^{(j)})$ is an arbitrary element of $\K$ with $\alpha^{(j)}_0=0$.
It follows that
$$L(z)=\frac{1}{1-zK(0,z)}=\frac{54z}{1+36z-(1-12z)^{3/2}}=\frac{1+36z+(1-12z)^{3/2}}{2(1+4z)^2}.$$

To find the asymptotic behavior of the coefficients, note that the singularity of $L(z)$ nearest to the origin is a branch point at $z=1/12$. By~\cite[Corollary~VI.1]{FS}, its coefficients satisfy
$$|\M_n(231)|\sim \frac{9}{32}\frac{n^{-5/2}}{\Gamma(-3/2)} 12^n$$ as claimed.
\end{proof}

We remark that expression~\eqref{eq:Ll0z} shows that $K(0,z)$ agrees with the generating function for planar rooted maps given in~\cite[Proposition~VII.11]{FS} and in the original paper by Tutte~\cite{Tutte}. In particular,
the coefficient of $z^n$ in $K(0,z)$ is $$\frac{2\,3^n\,(2n)!}{n!(n+2)!}.$$ This discovery has lead Bloom
to construct a bijection between the set of labeled Dyck
paths $(D,\alpha)\in\K_n$ with $\alpha_0=0$ and the set of planar rooted maps with $n$ edges~\cite{Bloommaps}.

It is interesting to observe that the generating function in Theorem~\ref{thm:312matchings} is algebraic, in contrast with the fact that the generating function for $123$-avoiding (namely, $3$-noncrossing) matchings is D-finite but not algebraic~\cite{Gou,CDDSY}. The first terms of these sequences, along with that of $132$-avoiding matchings, are given in Table~\ref{tab:matchings}.

Stankova~\cite{Stankova} showed that $|\M_F(231)|\le|\M_F(123)|\le|\M_F(132)|$ for any Ferrers board $F$, and she characterized for which boards $F$ the inequalities are strict.
Adding the first inequality over all $F\in\F_n$, we get $|\M_n(231)|\le|\M_n(123)|$ for all $n$, with strict inequality for $n\ge4$. The asymptotic estimate for $|\M_n(312)|$ given in
Theorem~\ref{thm:312matchings} allows us to quantify the magnitude of this inequality for large $n$. While the exponential growth rate of $|\M_n(231)|$ is $12$, as given by~\eqref{eq:asym312mat},
the growth rate of $|\M_n(123)|$ is 16, since
$$|\M_n(123)|=C_{n+2}C_n-C_{n+1}^2\sim\frac{24}{\pi n^5}\,16^n.$$

\subsection{$312$-avoiding partitions}

A refinement of the methods from Section~\ref{sec:312matchings} can be used to enumerate $312$-avoiding partitions, or equivalently, $231$-avoiding ones.

\begin{theorem}\label{thm:312partitions}
The generating function $B(z)=\sum_{n\ge 0}|\P_n(312)|z^n$ for $312$-avoiding partitions is a root of the cubic polynomial
\begin{multline}\label{eq:polHz}(z-1)(5z^2-2z+1)^2 B^3+(-9z^5+54z^4-85z^3+59z^2-14z+3)B^2\\ +(-9z^4+60z^3-64z^2+13z-3)B+(-9z^3+23z^2-4z+1).\end{multline}
The asymptotic behavior of its coefficients is given by
\beq\label{eq:asym312par}|\P_n(312)|\sim \delta n^{-5/2}\,\rho^{n},\eeq
where $$\rho=\frac{3(9+6\sqrt{3})^{1/3}}{2+2(9+6\sqrt{3})^{1/3}-(9+6\sqrt{3})^{2/3}}\approx6.97685$$
and $\delta\approx0.061518$.
\end{theorem}

\begin{proof}
To apply the argument presented in Section~\ref{sec:partitions}, we need to count $312$-avoiding matchings keeping track of the number of closers immediately followed by an opener, which we called valleys.
Via the bijection $\Pi\circ\kappa:\M_n(312)\rightarrow\L_n$, this is equivalent to counting labeled paths in $\L_n$ with
respect to the number of valleys.  We proceed as in the proof of Theorem~\ref{thm:312matchings}, but additionally keeping track of valleys.
We start by finding a generating function $K(u,v,z)$ for paths in $\K$ that refines $K(u,z)$ by marking the number of valleys with the the variable $v$.
The same argument that led to~\eqref{eq:Ll} now gives
\beq\label{eq:Llv}K(u,v,z)=1+z(vK(u,v,z)-v+1)\left(2K(u,v,z)+uK(u,v,z)+\frac{K(u,v,z)-K(0,v,z)}{u}\right),\eeq
since the number of valleys of $eD_1sD_2$ is the sum of the number of valleys of $D_1$ and $D_2$, plus one unless $D_2$ is empty.
Applying the quadratic method to~\eqref{eq:Llv} we obtain
$$K(0,v,z)=\frac{(4(v-1)^2z-2v-1)z S^2 + (8(z-1)^2z^2-2(2v+3)z+1) S + 4(v-1)^2z-2v-1)z}{4v(v-1)(S+1)z^2-2vz},$$
where $S=S(v,z)$ is a root of the cubic polynomial
$$4(v-1)^2z^2 S^3+(8(v-1)^2z-4v-5)z S^2+(4(v-1)^2z^2-2(2v+1)z+1)S-z.$$
To find the generating function $L(v,z)$ for paths in $\L_n$ that refines $L(z)$ by marking the number of valleys with $v$,
note that in the decomposition $D=eA_1seA_2s\dots$, a valley $se$ is created between each $A_i$ and $A_{i+1}$, so
$$L(v,z)=\frac{1/v}{1-vzK(0,v,z)}-\frac{1}{v}+1.$$
This generating function counts $312$-avoiding matchings with respect to the number of closers immediately followed by an opener.
To construct set partitions from matchings, each closer-opener pair can be independently merged into a transitory vertex or not.
Using Eq.~\eqref{eq:B}, the generating function for $312$-avoiding partitions is
$$B(z)=\frac{1}{1-z}L\left(\frac{1}{z},\frac{z^2}{(1-z)^2}\right)=\frac{c_2(z)R(z)^2+c_1(z)R(z)+c_0(z)}{(1-z)(3-7z)(1-2z+5z^2)^2},$$
where \begin{align*} c_2(z)&=4z^3(12z^3+7z^2-26z+3),\\ c_1(z)&=z^2(48z^4-179z^3+50z^2+81z-12),\\ c_0(z)&=48z^6-211z^5+270z^4-168z^3+73z^2-19z+3,\end{align*} and
$R(z)$ is a root of $4z^2R^3+(3z^2-4z)R^2+(3z^2-6z+1)R-z^2$. By computing the appropriate resultant, it follows that $B(z)$ is a root of the polynomial~\eqref{eq:polHz}.

To describe the asymptotic growth of its coefficients, we use the method described in~\cite[Section VII.7.1]{FS} to compute the singularities of algebraic functions. If $P(z,B)$ is the polynomial in Eq.~\eqref{eq:polHz},
its exceptional set is given by the zeroes of its discriminant $$z^9(z-1)(1-2z+5z^2)^2(27z^3-54z^2+63z-8)^3.$$
The dominant singularity is the exceptional point
$$\xi=-\frac{1}{3}(9+6\sqrt{3})^{1/3}+(9+6\sqrt{3})^{-1/3}+\frac{2}{3}\approx 0.14333,$$
which is the branch point closest of the origin.
For $z$ near $\xi$, the three branches of the cubic~\eqref{eq:polHz} give rise to one branch that is analytic with
value approximately $1.0368$ and a cycle of two conjugate branches with value $\alpha\approx1.2146$ at
$z = \xi$. The expansion of the two conjugate branches is of the singular type,
$$\alpha+\beta(\xi-z)\pm\gamma(\xi-z)^{3/2}+\dots,$$
where $\beta\approx-2.7077$ and $\gamma\approx2.6795$. It follows that
$$|\P_n(231)|\sim \gamma\xi^{1/2}\frac{n^{-5/2}}{\Gamma(-3/2)}\,\xi^{-n}=\frac{3\gamma}{4\sqrt{\pi n^5}}\,\xi^{-n+3/2}.$$
Letting $\rho=\xi^{-1}$ and $\delta=3\gamma\xi^{3/2}/(4\sqrt{\pi})$ we get the expression in the statement.
\end{proof}

Again, the generating function in Theorem~\ref{thm:312partitions} is algebraic, in contrast with the fact that the generating function for $123$-avoiding (namely, $3$-noncrossing) partitions is D-finite but not algebraic~\cite{BMXin}.
The first terms of the sequence are given in Table~\ref{tab:partitions}.

As in the case of matchings, from Stankova's  inequality $|\M_F(231)|\le|\M_F(123)|$ for any Ferrers board $F$ \cite{Stankova}, it follows that
$|\P_n(231)|\le|\P_n(123)|$ for all $n$. We can use Theorem~\ref{thm:312partitions} to see that the exponential growth
rates of $|\P_n(231)|$ are $|\P_n(123)|$ are different. While the growth rate of $|\P_n(231)|$ is $\rho\approx6.97685$, as given by~\eqref{eq:asym312par},
the growth rate of $|\P_n(123)|$ is 9. This was shown by Bousquet-M\'elou and Xin~\cite{BMXin}, who proved that the number of $3$-noncrossing partitions grows like
$$|\P_n(123)|\sim\frac{3^9\,5\sqrt{3}}{2^5\,\pi}\,\frac{9^n}{n^7}.$$

\subsection{An application to $1342$-avoiding permutations}\label{sec:Bona}

The method involving labeled Dyck paths that we have developed to enumerate $312$-avoiding matchings and partitions can be used to derive the following generating function due to B\'ona~\cite{Bona}
for the number of $1342$-avoiding permutations (which, by symmetry, equals the number of $3124$-avoiding ones).

\begin{theorem}[\cite{Bona}]
$$\sum_{n\geq0} |\S_n(1342)| z^n=\frac{32z}{1+20z-8z^2-(1-8z)^{3/2}}.$$
\end{theorem}

B\'ona~\cite{Bona} obtained this formula by constructing a bijection between so-called indecomposable $1342$-avoiding permutations and certain labeled trees, called
$\beta(0,1)$-trees. He then used the fact that the generating function for $\beta(0,1)$-trees had already been found by Tutte~\cite{Tutte}.
Our approach provides a more direct method to enumerate $1342$-avoiding permutations without using $\beta(0,1)$-trees.

We begin with a few straightforward definitions. We say that a rook placement $(R,F) \in \R_n$ is \emph{board minimal} if $F$ is the smallest Ferrers board that contains $R$.  Equivalently, $(R,F)$ is board minimal if and only if for every peak $V$ on $F$ the unit square of $F$ whose north-east vertex is $V$ belongs to $R$, i.e., it contains a rook. Denote by $\R_n^\times$ the set of placements in $\R_n$ that are board minimal, and similarly, for any pattern $\tau$, let
$\R_n(\tau)^\times$ be the the set board minimal elements of $\R_n(\tau)$.

We define a straightforward bijection $\chi: \S_n\to \R^\times_n$ as follows. For $\pi\in \S_n$, let $\chi(\pi)  =  (R_\pi, F_\pi)$ where $R_\pi$ is the placement consisting of the squares $(i, \pi(i))$ for $1\le i\le n$,
and $F_\pi$ is the smallest board containing $R_\pi$.  Observe that $\chi^{-1}(R,F) = \pi_R$, as defined in Section~\ref{sec:Ferrers}.
The following result will allow us to use our work on $312$-avoiding placements to obtain results about $3124$-avoiding (and thus $1342$-avoiding) permutations.

\begin{lemma} \label{lem:patterns and corners}
Let $\tau\in\S_k$ with $\tau(k) = k$ and $\tau(k-1) \neq k-1$.  Then $\chi$ restricts to a bijection between $\S_n(\tau)$ and $\R_n(\tau(1)\ldots \tau(k-1))^\times$.
\end{lemma}

\begin{proof}
Since $\tau(k) = k$, it is clear that $\chi(\S_n(\tau)) = \R_n(\tau)^\times$.  Therefore, it only remains to show that $\R_n(\tau)^\times = \R_n(\tau(1)\ldots \tau(k-1))^\times$.  The inclusion $\R_n(\tau(1)\ldots \tau(k-1))^\times\subset \R_n(\tau)^\times$ is trivial.  For the other inclusion, suppose that $(R,F) \notin \R_n(\tau(1)\ldots \tau(k-1))^\times$.  Then there must be a peak $V$ on $F$ such that the permutation determined by $R\,\cap\,\Gamma(V)$ contains an occurrence of $\tau(1)\ldots \tau(k-1)$.
Since $\tau(k-1) \neq k-1$, the rook in the unit square whose north-east vertex is $V$ (which is in $R$ because the placement is board minimal) is not part of this occurrence. Thus,
the occurrence of $\tau(1)\ldots \tau(k-1)$ together with this rook creates creates an occurrence of $\tau$, and so $(R,F)\notin\R_n(\tau)^\times$.
\end{proof}

Our next goal is to determine the image of the map $\Pi:\R_n(312)\to \L_n$, defined at the end of Section~\ref{sec:231background}, when restricted to $\R_n(312)^\times$.
We say that a labeled Dyck path $(D,\alpha)$ has the \emph{peak property} if for every peak $V_i$, the labels around it satisfy $\alpha_{i-1} = \alpha_i +1 = \alpha_{i+1}$.
Denote by $\Lu_n$ the set of labeled paths in $\L_n$ having the peak property.

\begin{lemma}\label{lem:b}
The map $\Pi:\R_n(312)\to \L_n$ restricts to a bijection between $\R_n(312)^\times$ and $\Lu_n$.
\end{lemma}

\begin{proof}
Recall that $\Pi$ maps $(R,F)$ to $(D_F,\alpha)$, where $\alpha_i$ is the length of a longest increasing sequence in $R\,\cap\,\Gamma(V_i)$.
Now observe that if $V_i$ is a peak on $F$, then the square in $F$ in this corner belongs to $R$ if and only if the longest increasing subsequence in $R\,\cap\,\Gamma(V_i)$ is one element longer
than the longest increasing subsequences in $R\,\cap\,\Gamma(V_{i-1})$ and $R\,\cap\,\Gamma(V_{i+1})$.
\end{proof}

Combining Lemmas~\ref{lem:patterns and corners} and~\ref{lem:b} we get a bijection between $\S_n(3124)$ and $\Lu_n$. We will derive a generating function for these paths using
the same framework that we used in Section~\ref{sec:312matchings} to enumerate paths in $\L_n$. 
Recall that in the proof of Theorem~\ref{thm:312matchings} we considered a larger set $\K_n$ consisting of labeled Dyck paths $(D,\alpha)$ that have the diagonal property and satisfy $\alpha_{2n}=0$.
Let now $\K^\times_n$ be the set of paths in $\K_n$ that have the peak property as well.
Note that $\Lu_n\subset \K^\times_n$.
Let $\K^\times=\bigcup_{n\ge 0}\K^\times_n$, and let
$$K^\times(u,z) = \sum_{n\ge 0} \sum_{(D,\alpha)\in \K^\times_n}u^{\alpha_0}z^n$$ be the generating function for such paths according to the value of the first label.

To obtain an equation for $K^\times(u,z)$, note that every nonempty $(D,\gamma)\in \K^\times$ can be decomposed uniquely as
$(D,\gamma)=(eD_1s,\alpha)\oplus (D_2, \beta)$ where $(eD_1s,\alpha),(D_2, \beta)\in \K^\times$. The labeling of $eD_1s$ can be of one of the four types in Figure~\ref{fig:DecompositionOfLabledPath}.
The difference with the proof of Theorem~\ref{thm:312matchings} is that now cases (b) and (c) can only occur if $D_1$ is nonempty, because otherwise $eD_1s$ would not have the peak property.
Consider the labeled Dyck path obtained by removing from $(eD_1s,\alpha)$ the initial $e$ and the final $s$ and, in cases (a) and (b), by decreasing by one the labels on $D_1$.
In case (a), this path is an arbitrary element of $\K^\times$; in cases (b) and (c), it is an arbitrary element of $\K^\times$ other than the empty path; in case (d), it is an arbitrary element of $\K^\times_n$
whose first label is nonzero.
This translates into the functional equation
$$K^\times(u,z) = 1+z K^\times(u,z) \left(K^\times(u,z)+u(K^\times(u,z)-1) + (K^\times(u,z)-1) + \frac{K^\times(u,z)-K^\times(0,z)}{u}\right).$$

Using the quadratic method to solve for $K^\times(0,z)$ we get
$$K^\times(0,z) = \frac{8z^2+12z-1+(1-8z)^{3/2}}{32z^2}.$$
Finally, to find the generating function for $\Lu_n$, observe that for any $(D,\alpha)\in\L_n$,
$D$ can be decomposed uniquely as $D=eA_1seA_2s\dots$, where each $A_j$ is a Dyck path, and if we let $\alpha^{(j)}$ be its sequence of labels decreased by one, then $(A_j,\alpha^{(j)})$ is an arbitrary element of $\K^\times$ with $\alpha^{(j)}_0=0$.
It follows that
$$\sum_{n\ge 0}|\S_n(3124)|z^n=\sum_{n\ge 0}|\Lu_n|z^n=\frac{1}{1-z K^\times(0,z)}=
\frac{32z}{1+20z-8z^2-(1-8z)^{3/2}}.$$

\section{The pattern $132$}\label{sec:132}

For last shape-Wilf-equivalence class of patterns of length 3, we have been unsuccessful in our attempts to apply our techniques
to find formulas for the generating functions
$$\sum_{n\geq 0} |\R_n(132)|x^n \quad \mbox{and} \quad \sum_{n\geq 0} |\P_n(132)|z^n.$$

\begin{question}
Find generating functions for $132$-avoiding matchings and partitions.
\end{question}

One reason to suspect that this may be a hard problem is that the related question of enumerating $\R_n(132)^{\times}$ ---that is, $132$-avoiding rook placements with a rook by each peak---
is equivalent to the outstanding open problem of enumerating $1324$-avoiding permutations~\cite{Bona,Bona2,CJS}, since $|\R_n(132)^{\times}| = |\S_n(1324)|$ by Lemma~\ref{lem:patterns and corners}.
We remark that in Section~\ref{sec:Bona}, a modification of our technique to enumerate $\R_n(312)$ allowed us to enumerate $\R_n^\times(312)$.


\section{Pairs of patterns}\label{sec:double}

In this section we investigate matchings and set partitions that avoid a pair of patterns of length~3.
For a pair of patterns $\sigma$ and $\tau$, we define the sets
$$\R_F(\sigma,\tau)=\R_F(\sigma)\cap\R_F(\tau),\qquad \M_F(\sigma,\tau)=\M_F(\sigma)\cap\M_F(\tau),\qquad \P_n(\sigma,\tau)=\P_n(\sigma)\cap\P_n(\tau),$$
$$\R_n(\sigma,\tau) = \bigcup_{F\in \F_n} \R_F(\sigma,\tau), \qquad \M_n(\sigma,\tau) = \bigcup_{F\in \F_n} \M_F(\sigma,\tau).$$
The notions defined in Section~\ref{sec:defs} for avoidance of one pattern have a straightforward generalization to sets of patterns.
For example, we say that two pairs of patterns are shape-Wilf-equivalent, which we write as $\{\sigma,\tau\}\sim\{\sigma',\tau'\}$, if
for any Ferrers board $F$ we have $|\R_F(\sigma,\tau)| = |\R_F(\sigma',\tau')|$ (equivalently, $|\M_F(\sigma,\tau)| = |\M_F(\sigma',\tau')|$).

We will establish that the $15$ pairs of patterns in $\S_3$ are partitioned into $7$ shape-Wilf-equivalence classes, as shown in Table~\ref{tab:shape wilf pairs}.
We provide enumeration results for matchings and set partitions avoiding a pair of patterns in all classes except for VI and VII. These results are summarized in Table~\ref{tab:enumeration of pairs}.  For matchings, the first few terms of the enumeration sequences for each class are given in Table~\ref{tab:Pairs count}.

\begin{table}[ht]
\begin{center}
\begin{tikzpicture}

\begin{scope}[shift={(9,0)}]
\tikzset{square matrix/.style={
    matrix of nodes,
    column sep=-\pgflinewidth, row sep=-\pgflinewidth,
    nodes={draw,
      minimum height=17,
      anchor=center,
      text width=27,
      align=center,
      inner sep=0pt
    },
  },
  square matrix/.default=.7cm
}
\matrix[square matrix]
{
|[fill = white]|&123&132&213&231&312&321\\
123&|[fill = lightgray]|&VI&I&II&III&IV\\
132&|[fill = lightgray]|&|[fill = lightgray]|&I&I&I&VII\\
213&|[fill = lightgray]|&|[fill = lightgray]|&|[fill = lightgray]|&I&I&V\\
231&|[fill = lightgray]|&|[fill = lightgray]|&|[fill = lightgray]|&|[fill = lightgray]|&I&I\\
312&|[fill = lightgray]|&|[fill = lightgray]|&|[fill = lightgray]|&|[fill = lightgray]|&|[fill = lightgray]|&I\\
321&|[fill = lightgray]|&|[fill = lightgray]|&|[fill = lightgray]|&|[fill = lightgray]|&|[fill = lightgray]|&|[fill = lightgray]|\\
};
\end{scope}

\begin{scope}[shift={(0,0)}]
\node at (0,0) {
\begin{tabular}{|c||c|}
\hline
Class & Shape-Wilf Equivalent Pairs\\
\hline
\multirow{3}{*}{I}& $\{123,213\} \sim \{132,213\} \sim \{132,231\}$\\
&$\sim \{132,312\} \sim \{213,231\}\sim \{213,312\}$\\
&$\sim \{231,312\} \sim \{231,321\} \sim \{312,321\}$
\\\hline
{II}&$\{123,231\}$\\\hline
{III}&$\{123,312\}$\\\hline
{IV}&$\{123,321\}$\\\hline
{V}&$\{213,321\}$\\\hline
{VI}&$\{123,132\}$\\\hline
{VII}&$\{132,321\}$\\\hline
\end{tabular}
};
\end{scope}
\end{tikzpicture}
\caption{Shape-Wilf equivalence classes for pairs of patterns of length 3.}
\label{tab:shape wilf pairs}
\end{center}
\end{table}

\begin{table}[ht]
\begin{center}
{\tabulinesep=1.2mm
\begin{tabu}{|c||c|c|}
\hline
Class&Matchings&Set partitions\\
\hline\hline
I&$\displaystyle\frac{4}{3+\sqrt{1-8z}}$& $\displaystyle\frac {2-3z+z^2 -z \sqrt {1-6z+z^2}}{2(1-3z+3z^2)}$\\
\hline
II \& III& Solutions of a cubic& Solutions of a cubic\\
\hline
IV&$\displaystyle\frac{1-5z+2z^2}{1-6z+5z^2}$& $\displaystyle\frac{1-10z+32z^2-37z^3+12z^4}{(1-z)(1-10z+31z^2-30z^3+z^4)}$\\
\hline
V& Functional equation&Unknown\\
\hline
VI \& VII & Unknown&Unknown\\
\hline
\end{tabu}}
\end{center}
\caption{A summary of the generating functions for matchings and set partitions avoiding a pair in each of the 7 classes.}
\label{tab:enumeration of pairs}
\end{table}

\begin{table}[ht!]
\begin{center}
\begin{tabular}{|c||c|c|c|c|c|c|c|}
\hline
\backslashbox{Class}{n} & 1&2&3&4&5&6&7\\
\hline\hline
I&1& 3 &13&67&381&2307&14589\\
\hline
II \& III& 1&3 &13&66&364&2112&12688\\
\hline
IV& 1&3 & 13 & 63 & 313&1563 &7813\\
\hline
V& 1&3 & 13 & 68 & 399 & 2528&16916\\
\hline
VI& 1 &3 & 13 & 69 &414 &2697&18625\\
\hline
VII &1 &3 &13 &66&363&2091&12407\\
\hline
\end{tabular}
\caption{The values of  $|\M_n(\sigma,\tau)|$ for $n\le 7$, for $\{\sigma,\tau\}$ in each of the 7 classes.}
\label{tab:Pairs count}
\end{center}
\end{table}

It will be convenient to introduce some well-known facts about Dyck paths. Let $C(v,z)$ be the generating function for Dyck paths with respect to the number of valleys, given by
$$C(v,z) = \sum_{n\ge 0}\sum_{D\in \D_n} v^{\val(D)} z^{n}= \frac{1-z+vz-\sqrt{1 - 2 \, {\left(v + 1\right)} z + (v-1)^2 z^{2} }}{2vz}.$$
The number of \emph{returns} of a Dyck path $D\in\D_n$, which we write as $r(D)$,
is defined to be the number of south steps of $D$ that intersect the diagonal $y=n-x$.
By keeping track of the number of valleys and the number of returns, we get the generating function
\beq\label{eq:Dyckrval}\sum_{n\ge 0}\sum_{D\in \D_n} t^{r(D)}v^{\val(D)} z^{n}= \frac{1+t(1-v)zC(v,z)}{1-tvzC(v,z)}.\eeq

In the next subsections we consider each one of the first five equivalence classes from Table~\ref{tab:shape wilf pairs}.
We will work in the context of rook placements on Ferrers boards, as the bijection $\kappa$ allows us to translate the results to matchings.
For each pair of patterns $\{\tau,\sigma\}$, we will use the fact established in Section~\ref{sec:partitions} that, once we find the generating function
$$A(v,z) = \sum_{n\ge 0}\sum_{M\in \M_n(\sigma,\tau)} v^{\val(M)} z^n =\sum_{n\ge 0}\sum_{F\in\F_n} |\M_F(\sigma,\tau)| v^{\val(F)} z^n
=\sum_{n\ge 0}\sum_{F\in\F_n} |\R_F(\sigma,\tau)| v^{\val(F)} z^n,$$ 
then we can obtain $\sum_{n\ge 0}|\M_n(\sigma,\tau)| z^n=A(1,z)$ and
\beq\label{eq:B2}\sum_{n\ge 0} |\P_n(\sigma,\tau)| z^n= \frac{1}{1-z}\, A\left(\frac{1}{z}, \frac{z^2}{(1-z)^2}\right),\eeq
by Eq.~\eqref{eq:B}.

\subsection{Equivalence class I}

For $F\in\F_n$, define its number of returns to be $r(F)=r(D_F)$. We now give the generating functions for matchings and partitions avoiding a pair of patterns from class I.

\begin{theorem}\label{thm:GF C(2,n)}
Let $\{\sigma,\tau\}$ 
be a pair in class I (see Table~\ref{tab:shape wilf pairs}). For all $F\in \F_n$, we have $$|\M_F(\sigma,\tau)| = 2^{n-r(F)},$$ and so all pairs in class I are shape-Wilf-equivalent.
Moreover,
$$\sum_{n\ge 0} |\M_n(\sigma,\tau)|z^{n} =  \frac{4}{3+\sqrt{1-8z}}, \quad\mbox{and so}\quad |\M_n(\sigma,\tau)| = \frac{1}{n+1}\sum_{k=0}^n \binom{2n+2}{n-k}\binom{n+k}{k},$$
and
$$\sum_{n\ge 0} |\P_n(\sigma,\tau)|z^{n} =  \frac {2-3z+z^2 -z \sqrt {1-6z+z^2}}{2(1-3z+3z^2)}.$$

\end{theorem}

\begin{proof}
Once we prove that $|\R_F(\sigma,\tau)| = 2^{n-r(F)}$ for all $F\in \F_n$, the rest follows easily. Indeed, it is then clear that the elements in class I are shape-Wilf-equivalent,
because the formula $2^{n-r(F)}$ only depends on $F$. The generating function for $\{\sigma,\tau\}$-avoiding matchings with respect to the number of valleys is then
\begin{multline}
A_1(v,z) := \sum_{n\ge 0}\sum_{M\in \M_n(\sigma,\tau)} v^{\val(M)} z^n
=\sum_{n\ge 0}\sum_{F\in\F_n} 2^{n-r(F)} v^{\val(F)} z^n
\\
=\sum_{n\ge 0}\sum_{D\in\D_n}2^{n - r(D)}v^{\val(D)} z^{n}= \frac{1+z(1-v)C(v,2z)}{1-vzC(v,2z)},\label{eq:A1}\end{multline}
where the last equality follows  by substituting $2z$ for $z$ and $1/2$ for $t$ in Eq.~\eqref{eq:Dyckrval}.
For $v=1$ the expression simplifies to
$$\sum_{n\ge 0} |\M_n(\sigma,\tau)|z^{n} = \frac{4}{3+\sqrt{1-8z}},$$
from where a formula for the coefficients $|\M_n(\sigma,\tau)|$ can be derived routinely.
The formula for $\{\sigma,\tau\}$-avoiding partitions follows by expanding Eqs.~\eqref{eq:B2} and~\eqref{eq:A1}.

It only remains to prove that $|\R_F(\sigma,\tau)| = 2^{n-r(F)}$ for all $F\in \F_n$, for every pair $\{\sigma,\tau\}$ in class I.
We first demonstrate a complete proof for the pair $\{\sigma,\tau\}=\{231,312\}$, and then show how simple modifications provide proofs for the other 8 pairs.

We proceed by induction on $n$, noting that the result is clear when $n=1$. Let $n\ge2$, and assume the result holds for boards in $\F_{n-1}$.
Let $F\in \F_n$, and let $F'\in\F_{n-1}$ be the board obtained from $F$ by removing its rightmost column and bottom row, and shifting the squares one unit down so that the bottom left corner of $F'$ is at the origin.
By induction hypothesis, $|\R_{F'} (231,312)| =  2^{n-1-r(F')}$.
Given a placement $R'$ such that $(R',F') \in \R_{F'}(231,312)$, we will consider all the placements $(R,F) \in \R_{F}(231,312)$ where $R$ is obtained from $R'$ by inserting a rook in the rightmost column of $F$,
and shifting up by one square the rooks in $R'$ in rows at least as high as that of the inserted rook. It is clear that all placements in $\R_F(231,312)$ arise in this fashion, since this process is reversible.
Let $k$ be the number of rows in the rightmost column of $F$, i.e., the number of south steps at the end of $D_F$.

If $k=1$, then any given $(R',F') \in \R_{F'}(231,312)$ gives rise to a unique placement in $\R_{F}(231,312)$, namely the one obtained from $R'$ by inserting a rook in position $(n,1)$ of $F$ and sliding all the other rooks up one square.
We conclude that $|\R_F(231,312)| = |\R_{F'} (231,312)| =2^{n-1-r(F')}=2^{n-r(F)}$, since $r(F') = r(F) -1$ in this case.

Suppose now that $k\ge 2$. Given $(R',F') \in \R_{F'}(231,312)$, we can always insert a rook in position $(n,k)$ of $F$ to obtain a placement in in $\R_{F}(231,312)$, since the inserted rook does not create occurrences of the patterns $231$ or $312$. Additionally, if the rightmost rook in $(R',F')$ is in position $(n-1,b)$,
a new rook can be inserted in position $(n,b)$ of $F$, giving rise to a $\{231,312\}$-avoiding placement. On the other hand, inserting a rook in position $(n,c)$ with $1\le c\leq b-1$ would create an occurrence of the pattern $231$,
while inserting a rook in position $(n,c)$ with $b<c<k$ would create an occurrence of the pattern $312$. It follows that
$$|\R_F(231,312)| =2\,|\R_{F'} (231,312)| =  2\cdot2^{n-1-r(F')} =2^{n-r(F)},$$
since $r(F') = r(F)$ in this case.

This inductive proof can be slightly modified for the following five additional pairs of patterns, by changing the available insertion locations for the new rook in the case $k\ge2$ as shown in the following table.

\bce
\bt{c|c}
Patterns&Available insertion locations for the new rook \\
\hline
$\{123,213\}$ & $(n,1)$ and $(n,2)$ \\
\hline
$\{231,321\}$ & $(n,k)$ and $(n,k-1)$ \\
\hline
$\{132,213\}$ &  $(n,b+1)$ and $(n,1)$\\
\hline
$\{213,231\}$ & $(n,b)$ and $(n,b+1)$\\
\hline
$\{132,312\}$ & $(n,k)$ and $(n,1)$
\et\ece

It remains to show that the three pairs $\{312,321\}$, $\{213,312\}$, and $\{132,231\}$ belong to this equivalence class.  For $(R,F)\in\R_n$, let $(R^t,F^t)$ be the placement obtained
by reflecting $(R,F)$ over the line $y=x$. Then, by symmetry, $(R,F)\in\R_F(312,321)$ if and only if $(R^t,F^t) \in \R_{F^t}(231,321)$. Since $r(F)=r(F^t)$, we have
$$|\R_F(312,321)| = |\R_{F^t}(231,321)| = 2^{n-r(F^t)} = 2^{n-r(F)},$$
as needed.  The argument for the other two pairs is analogous.
\end{proof}

\subsection{Equivalence classes II and III}\label{sec:IIandIII}

It is easy to see that the pairs $\{123,231\}$ and $\{123,312\}$ are not shape-Wilf equivalent. Indeed, a simple count reveals that if $F\in\F_5$ is the board that has three columns of height 5 followed by two columns of height 4,
then $|\R_F(123,231)| = 14$ and $|\R_F(123,312)| = 15$. However, as the following result shows, we have equality in the number of matchings avoiding each of the two pairs, and similarly in the number of partitions.
This answers Question~\ref{quest:converse} in the affirmative for pairs of patterns, although the original question remains open.

\begin{theorem}
Let $\{\sigma,\tau\}\in\{\{123,231\},\{123,312\}\}$. Then
$$\sum_{n\ge 0} |\M_n(\sigma,\tau)| z^n = \frac{1}{1-zH(1,z)},$$
and
$$\sum_{n\ge 0} |\P_n(\sigma,\tau)| z^n = 1+\frac{z(1-z)}{(1-z)^2-zH(v,z)},$$
where $H = H(v,z)$ is a rook of the cubic polynomial $\alpha H^3 + \beta H^2 + \gamma H + \delta$,
with
\begin{align*}
&\alpha = v^2z^2,\\
&\beta =  -2vz+2v(1- v)z^2 + v^2z^2C(v,z),\\
&\gamma = ( -2+3v) z + ( 1-v)^2 z^2 + \left( 1- (1+v)z+ v(1-v)z^2\right) C(v,z),\\
&\delta = -1+\left(vz-(1-v)^2z^2\right)C(v,z).\\
\end{align*}
\end{theorem}

\begin{proof}
Consider first the pair of patterns $\{123,312\}$. Recall from Section~\ref{sec:231} the bijection $\Pi:\R_F(312) \to \L_F$, that sends $(R,F)\in \R_F(312)$ to the labeled Dyck path $(D_F,\alpha)$, where $\alpha_i$ is the length of the longest increasing sequence in $R\,\cap\, \Gamma(V_i)$. Define $\L^{<3}_F$ to be the set of labeled paths $(D_F,\alpha)\in\L_F$ such that $\alpha_i<3$ for all $i$, and let $\L^{<3}_n=\bigcup_{F\in\F_n}\L^{<3}_F$. Then it is clear that $\Pi$ restricts to a bijection between $\R_F(123,312)$ and $\L_F^{<3}$. It follows that the generating function for $\{123,312\}$-avoiding matchings with respect to the number of valleys equals
$$A_2(v,z) := \sum_{n\ge 0}\sum_{F\in\F_n} |\R_F(123,312)| v^{\val(F)} z^n = \sum_{n\ge 0}\sum_{(D,\alpha)\in\L^{<3}_n} v^{\val(D)} z^n.$$
The generating function for $\{123,312\}$-avoiding partitions can be obtained from $A_2(v,z)$ using Eq.~\ref{eq:B2}.

To find an expression for $A_2(v,z)$, we follow the outline of the proofs of Theorems~\ref{thm:312matchings} and~\ref{thm:312partitions}.
Let $\K_n^{<2}$ be the set of labeled Dyck paths $(D,\alpha)$ of semilength $n$ that satisfy the diagonal property, have $\alpha_{2n} = 0$, and $\alpha_i<2$ for all $i$. Let $\K^{<2} = \bigcup_{n\ge 0} \K^{<2}_n$, and define
$$K^{<2}(u,v,z) = \sum_{n\ge 0}\sum_{(D,\alpha)\in \K^{<2}_n} u^{\alpha_0}v^{\val(D)}z^n.$$

To see the relationship between  $K^{<2}$ and $A_2$, observe that any element of $\L^{<3}$ may be uniquely constructed by taking paths $(D_1,\alpha^{(1)}),(D_2,\alpha^{(2)}),\ldots\in \K^{<2}$ with $\alpha^{(j)}_0=0$ for all $j$, and combining them into a path $(eD_1seD_2s\ldots eD_ks, \alpha)$, where the labels along each $D_i$ are given by the labels $\alpha^{(i)}$ incremented by 1, and the remaining labels are zero.  From this construction we see that
$$A_2(v,z) = \frac{1/v}{1-vzK^{<2}(0,v,z)}-\frac{1}{v}+1.$$

It remains to find an expression for $K^{<2}(0,v,z)$.
Using the operation $\oplus$ on $\K$ defined in the proof of Theorem~\ref{thm:312matchings}, every nonempty $(D,\gamma)\in\K^{<2}$ can be uniquely decomposed as
$(D,\gamma)=(eD_1s,\alpha)\oplus(D_2,\beta)$, where $(eD_1s,\alpha),(D_2,\beta)\in\K$. If $\alpha_j = 0$ for all $j$, then $(D_2, \beta)$ can be an arbitrary element of $\K^{<2}$.
If $\alpha_j=1$ for some $j$, then the condition $(D,\gamma)\in\K^{<2}$ forces $\beta_0 =0$, and we have the following possibilities for the labels $\alpha_1$ and $\alpha_{2i-1}$, where $i$ is the semilength of $eD_1s$:
\begin{itemize}
\item[(a)] $\alpha_{2i-1} = \alpha_1 =  0$, which forces $\alpha_0=0$; 
\item[(b)] $\alpha_{2i-1} = 0$ and $\alpha_1 = 1$, in which case $\alpha_0$ can be 0 or 1;
\item[(c)] $\alpha_{2i-1} = 1$, which forces $\alpha_1=1$ by the diagonal property, and thus $\alpha_j=1$ for $1\le j\le 2i-1$, but $\alpha_0$ can be 0 or 1.
\end{itemize}
Putting these cases together and noting that the generating function for labeled Dyck paths where all labels are 0 is $C(v,z)$,  we obtain the following function equation.
\begin{multline*}
K^{<2}(u,v,z) =  1 + zC(v,z)(vK^{<2}(u,v,z)-v+1)\\
+ z\Bigg(\left(K^{<2}(0,v,z)-C(v,z)\right)+ (K^{<2}(u,v,z) - K^{<2}(0,v,z)) \left(\frac{1}{u}+1\right)+(1+u)C(v,z)\Bigg)\\
(vK^{<2}(0,v,z)-v+1)
\end{multline*}
Writing $H=H(v,z)=K^{<2}(0,v,z)$ and $C=C(v,z)$, we obtain
\begin{equation}\label{eq:312 123 Functional}
\left(1-vzC-z(vH-v+1)\left(\frac{1}{u}+1\right) \right) K^{<2}(u,v,z)= 1+z(1-v)C + z(vH-v+1)\left(uC -\frac{H}{u}\right),
\end{equation}
which we solve using the kernel method. The expression multiplying $K^{<2}(u,v,z)$ in the left hand side is canceled by setting
$$u = \frac {z \left(1-v+vH \right) }{1-z+vz(1-C-H)}.$$
Making this substitution in Eq.~\eqref{eq:312 123 Functional}, we get
$$0 = \alpha H^3 + \beta H^2 + \gamma H + \delta,$$
with $\alpha,\beta,\gamma,\delta$ as in the statement.

Finally, to prove the that the generating functions for $\{123,231\}$-avoiding matchings and partitions are the same as for $\{123,312\}$-avoiding ones, it will suffice, by Eq.~\ref{eq:B2},
to show that
\begin{equation}\label{eq:231 312 Functional Equation Same}
\sum_{n\ge 0}\sum_{F\in\F_n} |\R_F(123,231)| v^{\val(F)} z^n = \sum_{n\ge 0}\sum_{F\in\F_n} |\R_F(123,312)| v^{\val(F)} z^n.
\end{equation}
Since $231$ and $312$ are inverses of each other, while $123$ is an involution, we have that $|\R_F(123,231)|= |\R_{F^t}(123,312)|$
for every $F\in\F_n$, where $F^t$ is the board obtained by reflecting $F$ over the line $y=x$.
Now Eq.~\eqref{eq:231 312 Functional Equation Same} follows using that $\val(F) = \val(F^t)$.
\end{proof}

\subsection{Equivalence class IV}

In order to state our results for the pair $\{123,321\}$, we need a few definitions.
 Let $D\in \D_n$ and let $h_0\ldots h_{2n}$ be its height sequence. We define the \emph{height} of $D$ to be $\max_{i} h_i$.
For $h\ge1$, denote by $\D_n^{<h}$ the set of paths in $\D_n$ whose height is strictly less than $h$, and by $\F_n^{<h}$ the set of Ferrers boards $F\in\F_n$ such that $D_F\in\D_n^{<h}$. Let $\D_n^{<h}=\bigcup_{n\ge0}\D^{<h}$.

The height of a peak $es$ formed by the $k$th and $(k+1)$st steps in $D$ is defined to be $h_k$.
Lastly, let $\eta(D)=|\{i: h_i = 2\}|$, and let $\eta(F)=\eta(D_F)$ for $F\in \F_n$.

\begin{theorem}\label{thm:123 321}
For all $F\in\F_n$ we have
$$|\M_F(123,321)| =\begin{cases}
 2^{\eta(F)} &\textrm{ if } F\in \F^{<5}_n,\\
 0&\textrm{otherwise.}
\end{cases}$$
Moreover,
$$\sum_{n\ge 0} |\M_n(123,321)|z^n =\frac{1-5z+2z^2}{1-6z+5z^2}, \quad\mbox{and so}\quad |\M_n(123,321)| = \frac{5^{n-1}+1}{2},$$
and
$$\sum_{n\leq 0} |\P_n(123,321)|z^n =\frac{1-10z+32z^2-37z^3+12z^4}{(1-z)(1-10z+31z^2-30z^3+z^4)}.$$
\end{theorem}

To prove this theorem we use the bijection $\delta_{321}: \R_F(321) \to \D^2_F$ defined in Section~\ref{sec:bij213}.
To describe the image of $\delta_{321}$ when restricted to $\R_F(123,321)$, we need one more definition.
Let $F\in \F_n^{<5}$, and let $h_0\ldots h_{2n}$ be the height sequence of $D_F$.
Define $\E_F^2$ to be the set of pairs $(D_0,D_F)\in \D_F^2$ where the height sequence $j_0\ldots j_{2n}$ of $D_0$ satisfies the following conditions:
if $h_i\in\{1,3\}$, then $j_i = 1$; if $h_i\in\{0,4\}$, then $j_i = 0$; and if $h_i = 2$, then $j_i\in\{0,2\}$.
Let $\E^2_n =  \bigcup_{F\in \F_n^{<5}} \E^2_F$.

\begin{lemma}\label{lem:power of two}
For any $F\in \D_n^{<5}$, $|\E_F^2| = 2^{\eta(F)}$.
\end{lemma}
\begin{proof}
Let $h_0\ldots h_{2n}$ be the height sequence of $D_F$. If $D_0$ is such that $(D_0,D_F)\in\E_F^2$, and $j_0\ldots j_{2n}$ is its height sequence,
then for each $i$ such that $h_i\neq 2$, the value of $j_i$ is determined by $h_i$.
Let now $i$ be such that $h_i=2$. Then $h_{i-1},h_{i+1}\in\{1,3\}$, which forces $j_{i-1}=j_{i+1}=1$, and so both choices of $j_i\in\{0,2\}$ are valid. It follows that
the number of choices of the height sequence $j_0\ldots j_{2n}$ is $2^{\eta(F)}$.
\end{proof}

The next lemma describes the image of the bijection $\delta_{321}$ on the set $\R_F(123,321)$.
\begin{lemma}\label{lem:image set 321 123}
If $F\in\F_n\setminus\F_n^{<5}$, then $\R_F(123,321)=\emptyset$. If $F\in F_n^{<5}$, then
$\delta_{321}$ restricts to a bijection between $\R_F(123,321)$ and $\E_F^2$.
\end{lemma}

\begin{proof}
Let $F\in\F_n$, let $V_0,\ldots,V_{2n}$ be the vertices along the border of $F$, and let $h_0\ldots h_{2n}$ be the height sequence of $D_F$.

Consider first the case that $F\notin\F_n^{<5}$. Suppose for contradiction that $\R_F(123,321)\neq\emptyset$, and let $R\in\R_F(321,123)$. Then, by Lemma~\ref{lem:height_sequence}, $h_i =|R \,\cap\, \Gamma(V_i)|$ for every $i$.
By the Erd\"os-Szekeres theorem, any permutation of length at least $5$ must have either an increasing or decreasing subsequence of length at least $3$.
Thus, if $i$ is such that $h_i\geq 5$, the permutation given by $R\,\cap\,\Gamma(V_i)$ contains an occurrence of $123$ or $321$, contradicting that $R\in R_F(123,321)$.

Now consider the case that $F\in \F_n^{<5}$. Recall from Section~\ref{sec:bij213} that for $(R,F)\in\R_F(123,321)$, the bottom path in $\delta_{321}(R,F)$ is defined as the one with height sequence $j_i = 2\ell_i - h_i$, where $\ell_i$ is the length of a longest increasing sequence in $R \,\cap\, \Gamma(V_i)$.
From this definition of $j_i$ and the fact that $0\leq j_i\leq h_i$ and $\ell_i<3$ for all $i$, it follows easily that  $j_i = 1$ when $h_i\in\{1,3\}$, $j_i = 0$ when $h_i\in\{0,4\}$, and $j_i\in\{0,2\}$ when $h_i = 2$.
This shows that $\delta_{321}(R,F) \in\E_F^2$, proving the inclusion $\delta_{321}(\R_F(123,321))\subseteq\E_F^2$.

To establish the reverse inclusion, let $D_0$ be any path such that $(D_0,D_F) \in \E_F^2$, and let $c_0\ldots c_{2n}$ be its height sequence.
Let $(R_0,F)=\delta_{321}^{-1}(D_0,D_F)\in \R_F(321)$, which exists because $\delta_{321}$ is a bijection between $\R_F(321)$ and $\D^2_F$.
Letting $\ell_i$ be the length of a longest increasing sequence in $R_0 \,\cap\, \Gamma(V_i)$, we know be definition of $\delta_{321}$ that $c_i=2\ell_i-h_i$. Thus, $\ell_i=(c_i+h_i)/2<3$,
where the last inequality follows from the fact that $(D_0,D_F) \in \E_F^2$. We conclude that $(R_0,F)\in \R_F(123,321)$.
\end{proof}

\begin{proof}[Proof of Theorem~\ref{thm:123 321}]
It follows from Lemmas~\ref{lem:power of two} and~\ref{lem:image set 321 123} that
\begin{equation}\label{eqn:321,123}
 |\R_F(123,321)|=|\E^2_F|=\begin{cases}
 2^{\eta(F)} &\textrm{ if } F\in \F^{<5}_n,\\
 0&\textrm{otherwise.}
\end{cases}
\end{equation}
This implies that
$$A_4(v,z) := \sum_{n\ge 0} \sum_{F\in\F_n} |\R_F(123,321)| v^{\val(F)}z^n = \sum_{n\ge 0} \sum_{D\in \D^{<5}_n}2^{\eta(D)}v^{\val(D)}z^n.$$
If we let
$$Q(u,v,z)=\sum_{n\ge 0} \sum_{D\in \D^{<5}_n} u^{\eta(D)}v^{\val(D)}z^n,$$
then $\sum_{n\ge 0} |\M_n(123,321)|z^n =Q(2,1,z)$ and $A_4(v,z)=Q(2,v,z)$, from where
$\sum_{n\leq 0} |\P_n(123,321)|z^n$ can be found using Eq.~\ref{eq:B2}.
Thus, it suffices to find an expression for $Q(u,v,z)$.

In the following generating functions for Dyck paths, $u$ marks the statistic $\eta$, and $v$ marks the statistic $\val$.
Let $$T(v,z)=\frac{1/v}{1-vz}-\frac{1}{v}+1$$ be the generating function for $\D^{<2}$. Since paths in $\D^{<3}$
can be obtained from paths in $\D^{<2}$ by inserting a path in $\D^{<2}$ at each peak, the generating function for $\D^{<3}$ is
$T(v,z\,T(v,uz))$. Similarly, $T(v,z\,T(v,uz\,T(v,uz)))$ is the generating function for $\D^{<4}$, obtained from paths in $\D^{<3}$ by inserting a path in $\D^{<2}$ at each peak (see Figure~\ref{fig:Dyck path height construction}).
Finally, inserting path in $\D^{<2}$ at the peaks of paths in $\D^{<4}$, we have
$$Q(u,v,z)=T(v,z\,T(v,uz\,T(v,uz\,T(v,z)))),$$
from where we obtain the stated generating functions for $\{123,321\}$-avoiding matchings and set partitions.
\end{proof}

\begin{figure}[ht]
\begin{center}
\begin{tikzpicture}[scale=0.4]
	\def\U{-- ++(1,0) [fill] circle(1.3pt)}
	\def\D{-- ++(0,-1) [fill] circle(1.3pt)}
	
\begin{scope}[shift = {(0,0)}]	
	\draw (0,0) -- (0,7);
	\draw (0,0) -- (7,0);
	\draw (0,6) circle(1.3pt) \U\U\D\U\D\D\U\D\U\U\D\D;
	\draw (0,6) -- (6,0);
\end{scope}

\begin{scope}[shift = {(8.5,4.5)}]
	\draw (0,0) -- (0,3);
	\draw (0,0) -- (3,0);
	\draw (0,2) circle(1.3pt) \U\D\U\D;
	\filldraw[fill=gray] (0,2) -- (1,2) -- (1,1) -- (2,1) -- (2,0) -- (0,2);
	\draw[->] (-.1,-.1) -- (-2,-2);
\end{scope}

\begin{scope}[shift = {(14,0)}]
	
	\draw (0,0) -- (0,9);
	\draw (0,0) -- (9,0);
	\draw (0,8) circle(1.3pt) \U\U\D\U\D\D\U\D\U\U\U\D\U\D\D\D;
	\draw (0,8) -- (8,0);
	\filldraw[fill=gray] (6,4) -- (7,4) -- (7,3) -- (8,3) -- (8,2) -- (6,4);
\end{scope}
\end{tikzpicture}
\caption{An example of the construction of a path in $\D^{<4}$ by inserting a path in $\D^{<2}$ at a peak of a path in $\D^{<3}$.}
\label{fig:Dyck path height construction}
\end{center}
\end{figure}
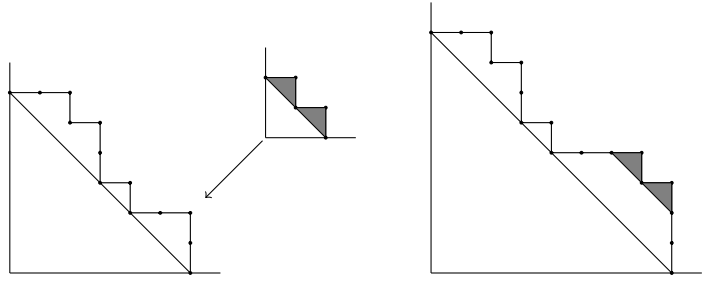

\subsection{Equivalence class V}

In this section we give a functional equation for the generating function for $\{213,321\}$-avoiding matchings, with two additional variables.
Unlike in the previous sections, we are unable to solve this equation to obtain an expression for the generating function.
We also do not give results on the enumeration of $\{213,321\}$-avoiding partitions.

\begin{theorem}\label{thm:213 321}
The generating function for $\{213,321\}$-avoiding matchings is given by $$\sum_{n\ge 0} |\M_n(213,321)|z^{2n}=G(0,0,z),$$ where $G(t,u,z)$ is the solution to the functional equation
\begin{multline*}
G(t,u,z) = 1 + z\Bigg(t\,G(t,u,z) +\frac{G(t,u,z)-G(t,0,z)}{tu}+\frac{G(t,0,z)- G(0,0,z)}{t}
\\
 + \frac{tu}{1-u}\,(G(t,0,z) - G(tu,0,z)) \Bigg).
\end{multline*}
\end{theorem}

To prove this theorem we use the bijection $\delta_{213}:\R_F(213)\rightarrow \D_F^2$ defined in Section~\ref{sec:bij213}. To describe its image when restricted to $\R_F(213,321)$, we define the following set.
For any $F\in \F_n$, denote by $\A^2_F$  the set of all pairs $(D_0,D_F)\in\D_F^2$ such that if $V$ is a vertex on the top path $D_F$,
then $D_0$ has no peak that lies strictly south and strictly west of $V$. As usual, we let $\A_n^2 = \bigcup_{F\in\F_n} \A_F^2$. Figure~\ref{fig:213_321 example} illustrates this definition.

\begin{figure}[h!]
\begin{center}
\begin{tikzpicture}[scale=0.4]

\begin{scope}[shift={(0,0)}]
\draw[fill=gray!20] (0,0) rectangle (2,8);
\draw[fill=gray!20] (2,0) rectangle (5,6);
\draw[fill=gray!20] (5,0) rectangle (8,4);

\draw (0,0) -- (0,8);
\draw (1,0) -- (1,8);
\draw (2,0) -- (2,8);
\draw (3,0) -- (3,8);
\draw (4,0) -- (4,7);
\draw (5,0) -- (5,6);
\draw (6,0) -- (6,5);
\draw (7,0) -- (7,5);
\draw (8,0) -- (8,4);

\draw (0,0) -- (8,0);
\draw (0,1) -- (8,1);
\draw (0,2) -- (8,2);
\draw (0,3) -- (8,3);
\draw (0,4) -- (8,4);
\draw (0,5) -- (7,5);
\draw (0,6) -- (5,6);
\draw (0,7) -- (4,7);
\draw (0,8) -- (3,8);

\def\R{-- ++(1,0) [fill] circle(1.3pt)}
\def\D{-- ++(0,-1) [fill] circle(1.3pt)}
\draw(0,8) circle(1.3pt)\R\R\R\D\R\D\R\D\R\R\D\R\D\D\D\D;
\draw (0,8) circle(1.3pt)\R\R\D\D\R\R\R\D\D\R\R\R\D\D\D\D;

\node at (-2.4,4.5) {$F_0$};
\draw[->](-1.8,4.7) -- (0.5,5.5);

\node at (8.5,6.7) {$F$};
\draw[->](8,6.3) -- (6.5,4.5);
\node at (1.5,7.5){$\times$};
\node at (0.5,6.5){$\times$};
\node at (4.5,5.5){$\times$};
\node at (3.5,4.5){$\times$};
\node at (7.5,3.5){$\times$};
\node at (6.5,2.5){$\times$};
\node at (5.5,1.5){$\times$};
\node at (2.5,0.5){$\times$};
\end{scope}

\begin{scope}[shift={(18,0)}]
\draw[fill=gray!20] (0,0) rectangle (2,8);
\draw[fill=gray!20] (0,0) rectangle (4,5);
\draw[fill=gray!20] (0,0) rectangle (8,4);
\draw[fill=gray!20] (0,0) rectangle (3,6);

\draw (0,0) -- (0,8);
\draw (1,0) -- (1,8);
\draw (2,0) -- (2,8);
\draw (3,0) -- (3,8);
\draw (4,0) -- (4,7);
\draw (5,0) -- (5,6);
\draw (6,0) -- (6,5);
\draw (7,0) -- (7,5);
\draw (8,0) -- (8,4);

\draw (0,0) -- (8,0);
\draw (0,1) -- (8,1);
\draw (0,2) -- (8,2);
\draw (0,3) -- (8,3);
\draw (0,4) -- (8,4);
\draw (0,5) -- (7,5);
\draw (0,6) -- (5,6);
\draw (0,7) -- (4,7);
\draw (0,8) -- (3,8);

\def\R{-- ++(1,0) [fill] circle(1.3pt)}
\def\D{-- ++(0,-1) [fill] circle(1.3pt)}
\draw(0,8) circle(1.3pt)\R\R\R\D\R\D\R\D\R\R\D\R\D\D\D\D;
\draw (0,8) circle(1.3pt)\R\R\D\D\R\D\R\D\R\R\R\R\D\D\D\D;

\node at (-2.4,4.5) {$F'_0$};
\draw[->](-1.8,4.7) -- (0.5,5.5);

\node at (8.5,6.7) {$F$};
\draw[->](8,6.3) -- (6.5,4.5);
\node at (1.5,7.5){$\times$};
\node at (0.5,6.5){$\times$};
\node at (2.5,5.5){$\times$};
\node at (3.5,4.5){$\times$};
\node at (7.5,3.5){$\times$};
\node at (6.5,2.5){$\times$};
\node at (5.5,1.5){$\times$};
\node at (4.5,0.5){$\times$};
\end{scope}
\end{tikzpicture}
\caption{On the left, the two Ferrers boards $F_0$ (shaded) and $F$ satisfy $(D_{F_0}, D_{F})\in \A_F^2$. On the right, $F_0'$ (shaded) and $F$ satisfy $(D_{F'_0}, D_{F})\notin \A_F^2$.   }
\label{fig:213_321 example}
\end{center}
\end{figure}
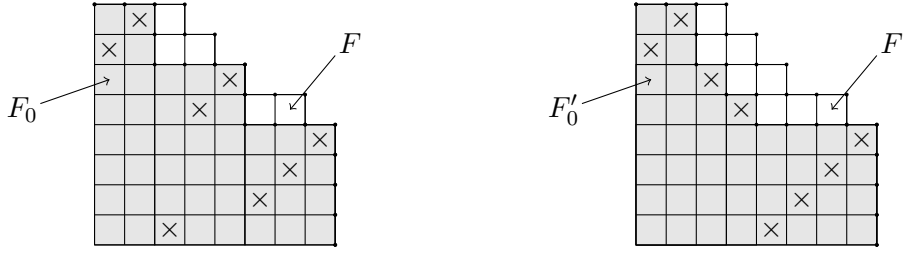

\begin{lemma}\label{lem:A2}
For $F\in\F_n$, the bijection $\delta_{213}$ restricts to a bijection between $\R_{F}(213,321)$ and  $\A^2_F$.
\end{lemma}

\begin{proof}
Let $(R,F)\in\R_F(213)$ and $\delta_{213}(R,F)=(D_0,D_F)\in\D^2_F$. Let $F_0\in \F_n$ be such that $D_{F_0}= D_0$.  Figure~\ref{fig:213_321 example} gives an example of this bijection both when $(D_0,D_F)\in \A^2_F$ (left) and when $(D_0,D_F)\notin \A^2_F$ (right).
Recall that, by definition of $\delta_{213}$, $F_0$ is the smallest Ferrers board that contains $R$. Consequently, if $(x,y)$ is a peak on $F_0$, then the unit square whose northeast vertex is $(x,y)$ contains a rook of $R$.  Furthermore, $R\, \cap\, \Gamma(x-1,y-1)$ is strictly increasing since $(R,F)\in\R_F(213)$.
We will show that $(R,F)\in\R_F(213,321)$ if and only if $(D_0,D_F)\in\A^2_F$.

First suppose that $(D_0,D_F)\notin \A^2_F$. Then there exists a vertex $V=(a,b)$ on the border of $F$ and a peak $(x,y)$ on $F_0$ such that $x<a$ and $y<b$.
The rook at this peak (i.e., in row $x-1$ and column $y-1$), together with the rook in row $b-1$ and the rook in column $a-1$, creates an occurrence of $321$ in $R\,\cap\,\Gamma(V)$,
and so $(R,F)\notin \R_F(213,321)$.

Now suppose that $(D_0,D_F)\in\A^2_F$. Then, for any vertex  $V=(a,b)$ on the border of $F$, the board $F_0$ has at most two peaks in the region $x\leq a$ and $y\leq b$.
It follows that the placement $R\,\cap\,\Gamma(V)$ is the union of two increasing sequence, so it avoids $321$. Therefore, $(R,F)\in\R_F(213,321)$.
\end{proof}

\begin{proof}[Proof of Theorem~\ref{thm:213 321}]
By Lemma~\ref{lem:A2},
$$\sum_{n\ge 0} |\M_n(213,321)|z^n = \sum_{n\geq 0} |\R_n(213,321)|z^n = \sum_{n\geq 0} |\A^2_n|z^n.$$
To enumerate $\A_n^2$, we will consider an larger set where we do not require paths to end on the diagonal.
It will also be convenient to shift the paths down by $n$ units.
More precisely, define $\B^2_n$ to be the set of pairs of lattice paths $(L_0,L_1)$ with steps south and east that start at the origin, remain above the line $y=-x$, and satisfy the following conditions:
\begin{itemize}
\item $L_1$ has a total of $n$ steps,
\item $L_0$ is weakly below $L_1$,
\item both paths end at the same $x$-coordinate,
\item $L_0$ has no peak strictly southwest of any vertex on $L_1$,
\item $L_0$ ends with a south step only if $L_1$ and $L_0$ end at the same $y$-coordinate.
\end{itemize}
Two examples of pairs in $\B^2_n$ are given in Figure~\ref{fig:marked board}.
Note that one may think of $\A^2_n$ as a subset of $ \B^2_{2n}$ by sliding the paths in $\A^2_n$ down by $n$ units.

\begin{figure}[htb]
\begin{center}
\begin{tikzpicture}[scale=0.4]
\begin{scope}[shift={(0,0)}]
\def\R{-- ++(1,0) [fill] circle(1.3pt)}
\def\D{-- ++(0,-1) [fill] circle(1.3pt)}
\draw(0.1,8.1) circle(1.3pt)\R\R\R\D\R\D\R\D\R\R\D\R\D\D;
\draw (0,8) circle(1.3pt)\R\R\D\D\R\R\R\D\D\R\R\D\D\D\R;
\end{scope}

\begin{scope}[shift={(14,0)}]
\def\R{-- ++(1,0) [fill] circle(1.3pt)}
\def\D{-- ++(0,-1) [fill] circle(1.3pt)}
\draw(0.1,8.1) circle(1.3pt)\R\R\R\D\R\D\R\D\R\R\D\R\D;
\draw (0,8) circle(1.3pt)      \R\R\D\D\R\R\R\D\D\R\R\R\D;
\end{scope}
\end{tikzpicture}
\caption{The example on the left is a pair in $\B_n^2$ where the bottom path ends with an east step.  In the example on the right the bottom path ends in a south step.}
\label{fig:marked board}
\end{center}
\end{figure}
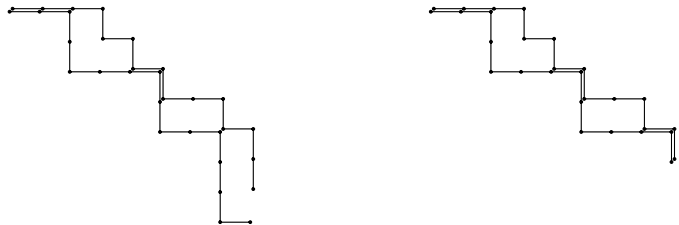

Define the following statistics on pairs $(L_0,L_1)\in \B^2_n$. If $V=(a,b)$ and $W=(a,d)$ are final vertices of $L_1$ and $L_0$, respectively, let $\hgt(L_1)= a+b$ and $\epsilon(L_0,L_1) = b-d$.
Consider the generating function
$$G(t,u,z) = \sum_{n\geq 0} \sum_{(L_0,L_1)\in\B^2_n} t^{\hgt(L_1)} u^{\epsilon(L_0,L_1)} z^n   = \sum_{i,j\geq 0}t^i u^jG_{i,j}(z),$$
and note that $G(0,0,z) = \sum_{n\geq 0} |\A^2_n|z^{2n}$.

To find a functional equation for $G$, consider, given $(L_0,L_1)\in \B^2_n$, all the ways that we may append steps at the end of $L_0$ and $L_1$ to obtain a pair in $\B^2_{n+1}$.  It follows from the definitions that we have the following options:
\begin{itemize}
\item[(a)] Add an east step to both $L_0$ and $L_1$.
\item[(b)] If $\epsilon(L_0,L_1)>0$, we may also add a south step to $L_1$ and leave $L_0$ unchanged.
\item[(c)] If $\epsilon(L_0,L_1)=0$ and $\hgt(L_1)>0$, we may also add either a south step to both $L_0$ and $L_1$, or one east step to $L_1$ and $\ell$ south steps followed by one east step to $L_0$, where $1\leq \ell\leq \hgt(L_1)$.
\end{itemize}
A routine calculation translates these rules into the following functional equation:
\begin{multline*}
G(t,u,z) = 1 + z\Bigg(t\,G(t,u,z) +\frac{G(t,u,z)-G(t,0,z)}{tu}+\frac{G(t,0,z)- G(0,0,z)}{t}\\
+ t\,\sum_{i>0} t^i(u+u^2+\ldots + u^i)G_{i,0}(z) \Bigg),
\end{multline*}
which yields the one in the statement after simplifying.
\end{proof}


\begin{thebibliography}{}

\bibitem{AERWZ} M. Albert, M. Elder, A. Rechnitzer, P. Westcott and M. Zabrocki, On the Stanley-Wilf limit of 4231-avoiding permutations and a conjecture of Arratia, {\it Adv. in Appl. Math.} \textbf{36} (2006), 96--105.


\bibitem{BWX} J. Backelin, J. West, G. Xin,
Wilf-equivalence for singleton classes, {\it Adv. Appl. Math.} \textbf{38} (2007), 133--148.


\bibitem{Bloommaps} J. Bloom, A bijection between planar maps and labeled Dyck paths, in preparation.

\bibitem{BloSar} J. Bloom, D. Saracino, A simple bijection between 231- avoiding and 312- avoiding placements, {\it  J. Combin. Math. Combin. Comput.}, to appear.

\bibitem{Bona} M. B\'ona, Exact enumeration of $1342$-avoiding permutations: a close link with labeled trees and planar maps, {\it J. Combin. Theory Ser. A} \textbf{80} (1997), 257–-272.

\bibitem{Bona2} M. B\'ona, A new upper bound for 1324-avoiding permutations, \texttt{arXiv:1207.2379}.

\bibitem{BMM} M. Bousquet-M\'elou, M. Mishna, Walks with small steps in the quarter plane, {\it Algorithmic probability and combinatorics} 1--39,
Contemp. Math., \textbf{520}, Amer. Math. Soc., Providence, RI, 2010.

\bibitem{BMSte}
M. Bousquet-M\'elou, E. Steingr\'imsson, Decreasing subsequences in permutations and Wilf equivalence for involutions, {\it J. Algebraic Combin.} \textbf{22} (2005), 383--409.

\bibitem{BMXin} M. Bousquet-M\'elou, G. Xin, On partitions avoiding 3-crossings, {\it S\'em. Lothar. Combin.} \textbf{54} (2005/07), Art. B54e, 21 pp.

\bibitem{BEMY} S. Burrill, S. Elizalde, M. Mishna and L. Yen, A generating tree approach to $k$-nonnesting partitions and permutations, preprint, \texttt{arXiv:1108.5615}.

\bibitem{CDDSY} W. Chen, E. Deng, R. Du, R. Stanley, C. Yan,
Crossings and nestings of matchings and partitions, {\it Trans. Amer. Math. Soc.} \textbf{359} (2007), 1555--1575.

\bibitem{CJS} A. Claesson, V. Jelínek, and E. Steingr\'{\i}msson, Upper bounds for the Stanley-Wilf limit of $1324$ and other layered patterns, {\it J. Combin. Theory Ser. A} \textbf{119} (2012), 1680–-1691.

\bibitem{Mier}
A. de Mier, $k$-noncrossing and $k$-nonnesting graphs and fillings of Ferrers diagrams, {\it Combinatorica} \textbf{27} (2007), 699-–720.

\bibitem{SC}
M. de Sainte-Catherine, {\it Couplages et Pfaffiens en combinatoire, physique et informatique}, Ph.D. Thesis, Universit\'e de Bordeaux I, 1983.

\bibitem{EliPak} S. Elizalde and I. Pak, Bijections for refined restricted permutations, {\it J. Combin. Theory Ser. A} \textbf{105} (2004), 207--219.

\bibitem{FS} P. Flajolet, R. Sedgewick, {\it Analytic Combinatorics}, Cambridge University Press, Cambridge, 2009.

\bibitem{Fom86}
S. Fomin, Generalized Robinson--Schensted-Knuth correspondence, {\it Zapiski Nauchn. Sem. LOMI} \textbf{155} (1986), 156--175.


\bibitem{Gou} D. Gouyou-Beauchamps, Standard Young tableaux of height 4 and 5, {\it European J. Combin.} \textbf{10} (1989), 69--82.

\bibitem{Goyt} A. M. Goyt, Avoidance of partitions of a three-element set, {\it Adv. in Appl. Math.} \textbf{41} (2008), 95--114.

\bibitem{Jel}
V. Jel\'inek, Dyck paths and pattern--avoiding matchings, {\it European J. Comb.} \textbf{28} (2007), 202--213.

\bibitem{JelMan} V. Jel\'inek, T. Mansour, On pattern-avoiding partitions, {\it Electron. J. Combin.} \textbf{15} (2008), \#39, 52 pp.

\bibitem{Kla} M. Klazar, On abab-free and abba-free set partitions, {\it European J. Combin.} \textbf{17} (1996) 53--68.

\bibitem{Kra01} C. Krattenthaler,   Permutations with restricted patterns and Dyck paths, {\it Adv. Appl. Math.} \textbf{27} (2001), 510-530.

\bibitem{Kra06} C. Krattenthaler, Growth diagrams, and increasing and decreasing chains in fillings of Ferrers shapes, {\it Adv. Appl. Math.} \textbf{37} (2006), 404--431.

\bibitem{ManSev} T. Mansour, S. Severini, Enumeration of $(k,2)$-noncrossing partitions,
{\it Discrete Math.} \textbf{308} (2008), 4570--4577.


\bibitem{OEIS} The On-Line Encyclopedia of Integer Sequences, published electronically at \texttt{http://oeis.org}.

\bibitem{Rio} J. Riordan, The distribution of crossings of chords joining pairs of $2n$ points on a circle, {\it Math. Computation} \textbf{29} (1975), 215--222.

\bibitem{Sag} B. E. Sagan, Pattern avoidance in set partitions, {\it Ars Combin.} \textbf{94} (2010), 79--96.



\bibitem{Sch}
C. Schensted, Longest increasing and decreasing subsequences, {\it Canad. J. Math} \textbf{13} (1961), 179--191.


\bibitem{Stankova} Z. Stankova, Shape-Wilf-Ordering on Permutations of Length 3, {\it Elec. J. Combin.} \textbf{14} (2007), \#R56.

\bibitem{SW} Z. Stankova and J. West, A new class of Wilf-equivalent permutations, {\it J. Alg.
Combin.} \textbf{15} (2002), 271--290.

\bibitem{EC2}
R. Stanley,  Enumerative Combinatorics, Vol. II, Cambridge Univ. Press, Cambridge, 1999.

\bibitem{Tou} J. Touchard, Sur un probl\`eme de configuration et sur les fractions continues, {\it Canad. J. Math.} \textbf{14} (1952), 2--25.

\bibitem{Tutte} W. T. Tutte, A census of planar maps, {\it Canad. J. Math.} \textbf{15} (1963), 249--271.

\end{thebibliography}
\end{document}